\documentclass[10pt,twoside,a4paper,reqno]{amsart}
\usepackage{amsfonts,amsthm,latexsym,amsmath,amssymb,amscd,epsfig}
\usepackage{graphics,graphicx}

\theoremstyle{plain}
\newtheorem{theo+}           {Theorem}
\newtheorem{prop+}           {Proposition}
\newtheorem{coro+}           {Corollary}
\newtheorem{lemm+}           {Lemma}

\newtheorem*{3-conj}         {The 3-Conjecture}

\theoremstyle{definition}
\newtheorem{defi+}           {Definition}
\newtheorem{problem}         {Problem}
\newtheorem*{pb}             {Problem}

\newtheorem*{pb-prime}       {Problem 1$\mathbb{'}$}

\newtheorem{case}            {Case}
\newtheorem{not+}            {Notation}

\theoremstyle{remark}
\newtheorem{rema+}           {Remark}
\newtheorem*{expl1}          {Explanations to Figure~\ref{fig1}}
\newtheorem*{example}        {Example}

\newenvironment{theorem}{\begin{theo+}}{\end{theo+}}
\newenvironment{proposition}{\begin{prop+}}{\end{prop+}}
\newenvironment{corollary}{\begin{coro+}}{\end{coro+}}
\newenvironment{lemma}{\begin{lemm+}}{\end{lemm+}}
\newenvironment{remark}{\begin{rema+}}{\end{rema+}}
\newenvironment{definition}{\begin{defi+}}{\end{defi+}}
\newenvironment{notation}{\begin{not+}}{\end{not+}}

\newcommand{\al}{\alpha}

\newcommand{\la}{\lambda}
\newcommand{\La} {\Lambda}
\newcommand{\Si}{\Sigma}
\newcommand{\bC}{\mathbb C}
\newcommand{\bR}{\mathbb R}
\newcommand{\bN}{\mathbb N}
\newcommand{\bZ}{\mathbb Z}

\newcommand{\C}{\mathcal C}
\newcommand{\calO}{\mathcal O}

\newcommand{\SG}{\mathcal{SG}}
\newcommand{\eps}{\epsilon}

 \def\bar{\overline}

 \def\0{{\bf 0}}

 \def\1{{\bf 1}}

 \def\o{{\bf o}}

 \def\u{{\bf u}}
 \def\v{{\bf v}}
 \def\w{{\bf w}}
 \def\x{{\bf x}}

 \def\y{{\bf y}}

 \def\hs{\hspace*{\parindent}}
 
 \def\gl{{\rm GL}}
 
 \newcommand{\lan}{\langle}
 \newcommand{\ran}{\rangle}

 \def\proof{{\hs \bf Proof.\ }}

 \def\t{^{\rm t}}

 \def\F{\mathord{\mathbb F}}
 \def\R{\mathord{\mathbb R}}
 \def\C{\mathord{\mathbb C}}

 \def\N{\mathord{\mathbb N}}
 \def\O{\mathord{\mathbb O}}
 \def\P{\mathord{\mathbb P}}

\numberwithin{equation}{section}

\begin{document}

\title[Parametric Poincar\'e-Perron theorem with applications]
{Parametric Poincar\'e-Perron theorem with applications}
\author[J.~Borcea]{$\text{Julius Borcea}^{\dagger}$}\thanks{$\dagger$ J.~B. unexpectedly passed away  on April 8, 2009 at the age of 40.  We dedicate  this paper (started jointly with J.~B. in Spring 2004) to the memory of this talented and tragic  human being. Rest in peace, Julius.}
\author[S.~Friedland]{Shmuel Friedland}
\address{
 Department of Mathematics, Statistics and
 Computer Science, 
 University of Illinois at Chicago, 
 Chicago, Illinois 60607-7045}
 \email{friedlan@uic.edu}
\author[B.~Shapiro]{$\text{Boris Shapiro}^{*}$}\thanks{$*$ Corresponding author}
\address{Department of Mathematics, Stockholm University, SE-106 91 Stockholm,
   Sweden}
\email{shapiro@math.su.se}

\begin{abstract}
We prove a parametric generalization of the classical Poincar\'e-Perron theorem on stabilizing recurrence relations
 where we assume  that the varying coefficients of a recurrence depend  on auxiliary parameters and converge uniformly 
 in these parameters to their limiting values.  As an application we study  convergence of the ratios  of families of functions 
 satisfying  finite recurrence relations with varying functional coefficients. For example, we explicitly  describe  the asymptotic ratio for two classes of biorthogonal polynomials introduced by Ismail and Masson.   \end{abstract}

\maketitle

\section{Introduction}\label{s1}

Consider  a usual linear recurrence relation of length $k+1$  with
constant coefficients
\begin{equation}\label{eq:Basic}
     u_{n+1}+\al_{1}u_{n}+\al_{2}u_{n-1}+\ldots+\al_{k}u_{n-k+1}=0,
\end{equation}
with $\al_{k}\neq 0$.

\begin{definition}\label{def4}
    The left-hand side of the equation
\begin{equation}\label{eq:Char}
	t^k+\al_{1}t^{k-1}+\al_{2}t^{k-2}+\ldots +\al_{k}=0
\end{equation}
is called the {\em characteristic polynomial} of
recurrence~\eqref{eq:Basic}. Denote the roots of~\eqref{eq:Char} by
$\la_{1},\ldots, \la_{k}$ and call them the {\em spectral numbers} of the
recurrence.
\end{definition}

The following simple theorem can  be found in e.g.~\cite[Ch.~4]{St}. 
 
\begin{theorem}\label{th:ordrec}
Let $k\in \bN$ and consider a $k$-tuple $(\al_{1},\ldots,\al_{k})$ of
complex numbers with $\al_{k}\neq 0$. For any function
$u:\bZ_+\to \bC$ the following conditions are equivalent:
\begin{enumerate}
\item[(i)] $\sum_{n\ge 0}u_{n}t^n=\frac {P(t)}{Q(t)}$,
where $Q(t)=1+\al_{1}t+\al_{2}t^2+\ldots+\al_{k}t^k$ and $P(t)$ is a
polynomial in $t$ whose degree is smaller than $k$.
\item[(ii)] For all $n\ge k-1$ the function $u_{n}$ satisfies  relation~\eqref{eq:Basic}.
\item[(iii)] For all $n\ge 0$ one has
\begin{equation}\label{eq:leadasymp}
	 u_{n}=\sum_{i=1}^rP_{i}(n)\la^n_{i},
\end{equation}
where $\la_{1},\ldots,\la_{r}$ are the distinct spectral numbers
of~\eqref{eq:Basic} with multiplicities $m_{1},\ldots,m_{r}$, respectively,
and $P_{i}(n)$ is a polynomial in the variable $n$ of degree at most
$m_{i}-1$ for $1\le i\le r$.
\end{enumerate}
\end{theorem}

The $k$-tuple $(u_0,...u_{k-1})$ which can be chosen arbitrarily is called the {\em initial $k$-tuple}. Denote the $k$-dimensional space of all initial $k$-tuples as $\bC^k$. 

\begin{definition}\label{def5}
Recurrence relation~\eqref{eq:Basic}  and its characteristic
polynomial~\eqref{eq:Char} are called {\em maxmod-generic}  if there exists a
unique and simple spectral number $\la_{max}$ of this recurrence satisfying
$|\la_{max}|=\max_{1\le i\le k}|\la_i|$.
Otherwise~\eqref{eq:Basic} and~\eqref{eq:Char} are called
{\em maxmod-nongeneric}.  The number
$\la_{max}$ will be referred to as the
{\em leading spectral number} of~\eqref{eq:Basic} or the {\em leading root}
of~\eqref{eq:Char}.
\end{definition}

\begin{definition}\label{df:slow}
    An initial $k$-tuple of complex numbers
     $(u_{0},u_{1},\ldots,u_{k-1})\in \bC^k$ is called {\em fast growing} with respect
to a given maxmod-generic recurrence \eqref{eq:Basic} if the
     coefficient $\kappa_{max}$ of $\la_{max}^n$ in~\eqref{eq:leadasymp} is
nonvanishing, that is, $u_{n}=\kappa_{max}\la_{max}^n+\ldots$ with
$\kappa_{max}\neq 0$.   Otherwise the $k$-tuple
$(u_{0},u_{1},\ldots,u_{k-1})$ is said to be {\em slow growing}. \end{definition}

\begin{remark}
Note that by Definition~\ref{def5} the leading spectral number $\la_{max}$ of any maxmod-generic recurrence has multiplicity one. An alternative characterization of  fast growing initial $k$-tuples is that they  have  the property $\lim_{n\to\infty}\frac{u_{n+1}}{u_n}=\la_{max}$.
One easily sees that  the set of all slowly growing initial $k$-tuples is a (complex) hyperplane in $\bC^k$,  its complement being the set of all fast growing $k$-tuples.
 The latter hyperplane of slow growing $k$-tuples can be found explicitly using linear algebra. 
\end{remark} 

\medskip 

A famous and frequently used  generalization of Theorem~\ref{th:ordrec}  in the case of variable coefficients was obtained by H.~Poincar\'e in 1885, \cite {Po} and later extended by O.~Perron, \cite{Pe}.  

\begin{theorem}[Poincar\'e-Perron] \label{th:Po} If the coefficients $\al_{i,n},\quad  i=1,...,k$ of a linear homogeneous difference equation 
\begin{equation}\label{eq:Po}
u_{n+k}+\al_{1,n}u_{n+k-1}+\al_{2,n}u_{n+k-2}+...+\al_{k,n}u_n=0
\end{equation}
have limits $\lim_{n\to \infty}\al_{i,n}=\al_i,\quad i=1,...,k$ and if the roots $\la_1,...,\la_k$ of the characteristic equation 
$t^k+\al_{1}t^{k-1}+...+\al_k=0$ have distinct absolute values then 
\begin{itemize}
\item[(i)]  for any solution $u$ of \eqref{eq:Po} either $u(n)=0$ for all sufficiently large $n$ or $\lim_{n\to\infty}\frac {u(n+1)}{u(n)}$ for $n\to \infty$ equals one of the roots of the characteristic equation. 
\item[(ii)]  if additionally $\al_{k,n}\neq 0$ for all $n$ then for every $\la_i$ there exists a solution $u$ of \eqref{eq:Po} with  $\lim_{n\to\infty}\frac {u(n+1)}{u(n)}=\la_i.$
\end{itemize} 
\end{theorem}Ê

\begin{remark}Ê
If as above  $\la_{max}$ will denote  the root of the limiting characteristic equation with the  maximal absolute value then under the assumptions  of Theorem~\ref{th:Po} (ii)  the set of solutions of \eqref{eq:Po} Êfor which $\lim_{x\to\infty}\frac {u(n+1)}{u(n)}\neq \la_{max}$  is a complex hyperplane in the space of all solutions.    For the latter fact to hold the assumption that all $\la_i$'s have distinct absolute values can be substituted by the weaker assumption of maxmod-genericity of the limiting recurrence, see details in Lemma~\ref{lm:varcoef}. But, in general,   there seems to be no  easy way to determine this hyperplane explicitly.
\end{remark}

A number of generalizations and applications of the  Poincar\'e-Perron theorem can be found in the literature, see e.g. \cite{Ko1}, \cite{Ko2}, \cite{MaNe}, \cite{MaMu}, \cite{Pi}Ê and references therein.  The set-up of Poincar\'e-Perron is often generalized 
to the case of   Poincar\'e difference systems. Namely, consider an iteration scheme: 
\begin{equation}\label{eq:Poin}
{\bf u}(n+1)=\left[A+B(n)\right]{\bf u}(n),
\end{equation}
where ${\bf u}(n)$ is a vector in $\bC^k$, $A$ and $B(n),\;n=0,1,...$ are $k\times k$-matrices such that 
$||B(n)||\to 0$ as $n\to \infty$. For example, one of quite recent results in this direction having a very strong resemblance with Theorem~\ref{th:Po} is as follows, see \cite{Pi}, Theorem~1.

\begin{theorem}\label{th:pit} Let $\bf u$ be a solution of  the system~\eqref{eq:Poin}  under the assumption that $||B(n)||\to 0$ as $n\to \infty$. Then either ${\bf u}(n)=0$ for all sufficiently large $n$ or 
$$\rho=\lim_{n\to \infty} \root  n \of {||u(n)||}$$ 
exists and is equal to the modulus of one of the eigenvalues of the matrix $A$. 
\end{theorem} 
 
In connection with the present project  the second author  earlier  obtained the following generalization of the Poincare-Perron theorem for the case of Poincar\'e difference systems,  see \cite[Theorem 1.2]{Fr1}. This statement  apparently covers the majority of results in this direction known at present. Let $\rm{M}_k(\F)$ and $\rm{GL}_k(\F)$ denote the spaces of all and  resp. of all invertible  $(k\times k)$-matrices over a field $\F$.

\begin{theorem}\label{th:linop} 
Let $\{T_{n}\}_{n \in \bN}$ be a sequence of regular matrices
in $\text{{\rm GL}}_{k}(\bC)$ converging to some (possibly singular)
matrix $T \in \text{{\rm M}}_{k}(\bC)$. Assume furthermore that $T$
has a positive spectral radius $\rho(T)$ and that the circle
$\big\{z\in \bC\,\big|\, \vert z \vert=\rho(T)\big\}$ contains exactly one
eigenvalue $\la_{max}$ of $T$
which is a simple root of its characteristic equation. Let ${\bf u}_{max}$
denote an eigenvector of $T$ corresponding to $\la_{max}$, i.e.,
$T{\bf u}_{max}=\la_{max}{\bf u}_{max}$, ${\bf 0}\neq{\bf u}_{max}\in \bC^{k}$. Then the complex line
spanned by
the product $T_{n}T_{n-1}\cdots T_{1}\in \text{{\rm M}}_{k}(\bC)$ converges to
the complex line spanned by ${\bf u}_{max}{\bf w}^{t} \in \text{{\rm M}}_{k}(\bC)$ for
some fixed vector ${\bf 0}\neq {\bf w}\in \bC^{k}$. Hence for any vector
${\bf x}_{0}\in \bC^{k}$ such that ${\bf w}^{t}{\bf x}_{0}\neq 0$ the complex
line in $\bC^k$ spanned by $T_{n}T_{n-1}\cdots T_{1}{\bf x}_{0}$
converges to the complex line in $\bC^k$ spanned by ${\bf u}_{max}$ as $n\to\infty$.
\end{theorem}

\begin{remark} Similarly to the situation with  Theorem~\ref{th:Po}   there seems to be no  easy way to explicitly determine the vector $\bf w$ in Theorem~\ref{th:linop}. \end{remark}Ê   

\medskip 
The  goal of the present paper  is to present an extension of Theorem~\ref{th:linop} for  sequences of invertible matrices depending continuously or analytically on auxiliary  parameters.  In other words, we are looking for a {\em parametric Poincar\'e-Perron theorem}.   The main  result of the present paper is as follows. 
 \begin{theorem}\label{anconvprod}  Let $\{T_n(\x)\}_{n\in\N}$ be a sequence of families of regular matrices in $\text{{\rm GL}}_{k}(\bC)$  depending continuously on $\x \in \mathcal{D}\subset\R^d$.  Assume that this sequence converges uniformly, on any compact set in $\mathcal{D}$,
 to a matrix $T(\x)$.  
 Suppose furthermore that for each $\x\in\mathcal{D}$, $T(\x)$ has exactly one simple
 eigenvalue $\lambda_{max}(\x)$ of the maximal modulus with the
 corresponding eigenvector $\hat \u_{max}(\x)\in \P^{k-1}$.
 Then the product $T_n(\x) T_{n-1}(\x)\ldots T_2(\x) T_1(\x)$,
 viewed as an automorphism of $\P^{k-1}$, converges to the transformation $\hat \u_{max}(\x)\hat \w\t(\x)$ on $\P^{k-1}\backslash \widehat{H(\x)}$,
 where $\hat \w(\x)\in\P^{k-1}$ is continuous in $\mathcal{D}$ and $H(\x)$ is the hyperplane given by $\w(\x)\t \v=0$. (Here $\P^{k-1}$ 
 stands for the complex projective space of dimension $k-1$.) 
 \end{theorem}

\begin{remark} In Theorem~\ref{anconvprod}  we consider continuous dependence of $\{T_n(\x)\}$ on $\x$ and the uniform convergence of this sequence to the limiting $T(\x)$. The same result holds in the analytic category if  $\{T_n(\x)\}$ depend analytically  on $\x$ and converge uniformly to $T(\x)$. One can also get the same result in the smooth category under the assumption that  $\{T_n(\x)\}$ depend smoothly on $\x$ and converge uniformly to $T(\x)$ together with their partial  derivatives of all orders. Since we consider families of complex matrices depending continuously on parameters we decided to use $\bR^d$ as a parameter space. The case when parameters  
belong to $\bC^d$ would be  equally natural.  
\end{remark}Ê

\begin{remark}ÊIn a sense Theorems~\ref{th:linop} and ~\ref{anconvprod} can be considered as a far reaching and parametric generalization of the well-known power method in linear algebra which is a simple iterative procedure allowing one to  determine the dominating eigenvalue and eigenvector of a given square matrix possessing a unique eigenvalue with maximal modulus, see e.g. \cite {Vo}. On the other hand, the fact that we consider sequences of matrices (instead of one and the same matrix) additionally depending on extra parameters creates substantial technical difficulties. 
\end{remark} 

 In spite of its simple  formulation it seems that Theorem~\ref{anconvprod}  has  no prototypes  in the existing literature.   We can now apply Theorems~\ref{th:Po}, ~\ref{th:linop}  and ~\ref{anconvprod}  to sequences of functions satisfying  finite linear  recurrence relations. The set-up is as follows. 

\begin{pb}
Given a positive integer $k\ge 2$ let  $\{\phi_{i,n}(\x)\}$, $1\le i\le k, n\in \bZ_+$, be 
a sequence of $k$-tuples of 
complex-valued functions of a  (multi-)variable $\x=(x_1,...,x_d)$ defined in some domain
$\Omega \subseteq \bR^d$. Describe the {\em asymptotics} when $n\to \infty$ of the ratio $\Psi_n(\x)=\frac{f_{n+1}(\x)}{f_{n}(\x)}$ for a  family
of complex-valued functions $\{f_{n}(\x)\mid n\in \bZ_{+}\}$
satisfying
\begin{equation}\label{eq:General}
f_{n+k}(\x)+\sum_{i=1}^{k}\phi_{i,n}(\x)f_{n+k-i}(\x)=0,\quad n\ge k-1.
\end{equation}
In other words, given a family $\{f_{n}(\x)\}$ of functions satisfying
 ~\eqref{eq:General} calculate the
asymptotic ratio $\Psi(\x)=\lim_{n\to \infty} \Psi_n(\x)$  (if it exists).
\end{pb}

To formulate our further results we need some notions. 
Denote by  $Pol_{k}=\left\{t^k+a_{1}t^{k-1}+\ldots +a_{k}\mid a_i\in \bC,
1\le i\le k\right\}$
 the set of all monic  polynomials of degree $k$ with complex coefficients. 

\begin{definition}\label{def7}
    The subset\ $\Xi_{k}\subset Pol_{k}$
    consisting of all maxmod-nongeneric polynomials is called the
     {\em standard maxmod-discriminant}, see Definition~2. For any family
$$\Gamma(t,\al_{1},\ldots,\al_{q})
=\left\{t^k+a_{1}(\al_{1},\ldots,\al_{q})t^{k-1}+\ldots
+a_{k}(\al_{1},\ldots,\al_{q})\right\}$$
of monic  polynomials of degree $k$ in $t$ we define the {\em induced maxmod-discriminant} $\Xi_{\Gamma}$ to be the set of all parameter values
    $(\al_{1},\ldots,\al_{q})\in \bC^q$ for which the corresponding polynomial
in $t$ is maxmod-nongeneric, i.e. belongs to $\Xi_{k}$.
\end{definition}

Some local properties of $\Xi_{k}$ can also be derived from \cite{Br} and \cite{Ma}.

\begin{example}
For $k=2$ the maxmod-discriminant $\Xi_{2}\subset
    Pol_{2}$  is the real hypersurface consisting of the set of all pairs $(a_1, a_2)$ such that there exists $\eps \in [1,\infty)$ solving the equation  $\eps a_{1}^2-4 a_{2}=0$, see Lemma~\ref{lm:deg2}.  More
    information on $\Xi_{k}$  will be given in \S 6.
\end{example}

Now take   a family 
$$\left\{\overline \phi_{n}
:=\big(\phi_{1,n}(\x),\phi_{2,n}(\x),
   \ldots,\phi_{k,n}(\x)\big)\mid n\in \bZ_+\right\}, \; \x=(x_1,...,x_d)$$
   of $k$-tuples of complex-valued functions defined on some 
   domain $\Omega \subset \bR^d$ such that 
   
   (i) $\phi_{k,n}(\x)$ is non-vanishing in $\Omega$;
   
   (ii)    $\overline \phi_{n}$ converges  to a fixed $k$-tuple of functions
$\widetilde \phi=\Big(\widetilde \phi_{1}(\x), \widetilde \phi_{2}(\x),
   \ldots,\widetilde \phi_{k}(\x)\Big)$ pointwise in $\Omega$. 
   
Choose some initial $k$-tuple of 
functions $I=\big(f_{0}(\x),\ldots,
   f_{k-1}(\x)\big)$  defined on $\Omega$ and determine the 
 family of  functions  $\{f_{n}(\x)\mid n\in \bZ_+\}$   that satisfies the recurrence relation \eqref{eq:General}
for all $n\ge k$ and coincides with
$I=\big(f_{0},\ldots,f_{k-1}\big)$ for $0\le n\le k-1$.

\medskip 
Theorems~\ref{th:Po}, ~\ref{th:linop} and \ref{anconvprod}  imply the following. 

\begin{theorem}\label{th:Exist}
In the above notation there exists a unique subset
$\Si_{I}\subseteq \Omega\setminus \Xi_{\widetilde\phi}$ which is minimal with
respect to inclusion and such that the following holds:
\begin{enumerate}
\item[(i)] For any $\x\in 
\Omega \setminus (\Xi_{\widetilde\phi}\cup \Si_{I})$ one has
$$\lim_{n\to \infty}\frac{f_{n+1}(\x)}
{f_{n}(\x)}=\Psi_{max}(\x),$$
where $\Psi_{max}(\x)$ is the leading root of the 
{\em asymptotic symbol equation}
\begin{equation}\label{eq:Symb}
	\Psi^k+\widetilde \phi_{1}(\x)\Psi^{k-1}+\widetilde
	\phi_{2}(\x)\Psi^{k-2}+\ldots+\widetilde \phi_{k}(\x)=0
\end{equation}
and $\,\Xi_{\widetilde\phi}$ denotes the induced maxmod-discriminant of~\eqref{eq:Symb} considered as a family of monic polynomials
in the variable $\Psi$ (cf.~Definition~\ref{def7}).  
\item[(ii)] If $\overline \phi_{n}$ consists of continuous functions and $\overline \phi_{n}\rightrightarrows 
\widetilde\phi=\big(\widetilde \phi_{1}(\x),\widetilde \phi_{2}(\x),\ldots,\widetilde \phi_{k}(\x)\big)$
in $\Omega$ then
$$\frac{f_{n+1}(\x)}
{f_{n}(\x)}\rightrightarrows 
\Psi_{max}(\x)\,\text{ in }\,\Omega \setminus 
(\Xi_{\widetilde\phi}\cup \Si_{I}),$$
where $\rightrightarrows$ stands  for uniform convergence on 
 compact subsets in $\Omega$.
\end{enumerate}
\end{theorem}

\begin{remark} Notice that since we assumed that  convergence in (ii) is uniform  then each 
$\widetilde \phi_{i}(\x)$ is continuous in $\Omega$. Moreover, since a uniformly converging sequence of  analytic functions necessarily converges to a analytic  function one has that (ii) holds in the  analytic category as well. To get the latter result in the smooth category one should require the uniform convergence   
$\overline \phi_{n}\rightrightarrows \widetilde\phi$ together with their partial derivatives of all orders.
\end{remark} 

\begin{definition}
The set $\Si_{I}$ introduced in Theorem~\ref{th:Exist} is called the 
{\em set of slowly growing initial conditions}
(cf.~Definition~\ref{df:slow}). If the functional coefficients
in~\eqref{eq:General} are fixed (i.e., independent of $n$) then
        $\Si_{I}$ is exactly the set of all points $p\in \Omega$ such that
        the initial $k$-tuple $I(p)=\big(f_{0}(p),\ldots,f_{k-1}(p)\big)$ is slowly
        growing with respect to the recurrence~\eqref{eq:General} evaluated at
$p$.  
\end{definition}

\begin{remark}
The exact description of $\Si_{I}$ in the case of varying coefficients can
be found in the proof of Theorem~\ref{th:Exist}, see \S 2. Note that 
convergence of $\frac{f_{n+1}(\x)}{f_{n}(\x)}$ to the leading root of~\eqref{eq:Symb}
might actually occur at some points $\x$ lying in $\Xi_{\widetilde\phi}$ where the leading
root is still unique but multiple.
\end{remark}

When both the coefficients
$\phi_{i,n}(z)$ and the functions $f_{n}(z)$ are complex analytic one can 
associate
to each meromorphic ratio $\frac{f_{n+1}(z)}{f_{n}(z)}$ the following useful
complex-valued distribution $\nu_{n}$, see \cite[p.~249]{BeG}. (Since we will only use this distribution in one-dimensional case we will define it for $\bC$. For the multi-dimensional case consult  e.g. \cite {Ts}.)

\begin{definition}\label{def1} 
Given a  meromorphic function $g$ in some open set $\Omega\subseteq \bC$ we
construct its (complex-valued) {\em residue distribution} $\nu_{g}$ as
follows. Let $\{z_{m}\mid m\in \bN\}$ be the (finite or infinite) set of all
the poles of $g$ in $\Omega$. Assume that the Laurent
     expansion of $g$ at $z_{m}$ has the form $g(z)=\sum_{-\infty<l\le
     l_{m}}\frac{T_{m,l}}{(z-z_{m})^{l}}$.
     Then the distribution $\nu_{g}$ is given by
\begin{equation}\label{resid}
     \nu_{g}=\sum_{m\ge 1}\left(\sum_{1\le l\le
     l_{m}}\frac{(-1)^{l-1}}{(l-1)!}T_{m,l}\frac{\partial^{l-1}}{\partial
     z^{l-1}}\delta_{z_{m}}\right),
\end{equation}
     where $\delta_{z_{m}}$ is the Dirac mass at $z_{m}$. The above 
sum is meaningful as a distribution on $\Omega$ since it is
     locally finite there.
\end{definition}

\begin{remark}
The distribution $\nu_{g}$ is a complex-valued {\em measure}
if and only if $g$ has all simple poles, see \cite[p.~250]{BeG}. If the
latter holds then in the notation of Definition~\ref{def1} the value of this
complex measure at $z_{m}$ equals $T_{m,1}$, i.e., the residue of $g$ at
$z_{m}$.
\end{remark}




\begin{definition}\label{def3}
If $\{f_{n}(z)\}$ consists of functions which are analytic in $\Omega$ 
and $\nu_n$ denotes the
residue distribution of the meromorphic function
$\frac{f_{n+1}(z)}{f_{n}(z)}$ in $\Omega$ then the limit
    $\nu=\lim_{n\to\infty}\nu_{n}$ (if it exists in the sense
    of weak convergence) is called  the
{\em asymptotic ratio distribution} of the family $\{f_{n}(z)\}$.
\end {definition}

\begin{remark} Notice that the support of $\nu$ describes the asymptotics of the  zero loci of the family $\{f_n(z)\}$. 

\end{remark}Ê

\begin{proposition}\label{pr:2} Under the assumptions of Theorem~\ref{th:Exist} (ii) 
 in the complex analytic category the following holds:
\begin{enumerate}
\item[(i)]  The support of $\nu$ belongs to $\Xi_{\widetilde\phi}$, where, as
before,  $\Xi_{\widetilde\phi}$ denotes  the induced maxmod-discriminant. The set $\Sigma_{I}$ of slowly
growing initial conditions is irrelevant to the support of $\nu$. 
\item[(ii)] Suppose that there exists a nonisolated  
point $p_0\in \Xi_{\widetilde\phi}$ such that equation \eqref{eq:Symb} considered at 
$p_0$ has the property that among its roots with maximal 
absolute value there are at least two with the same maximal multiplicity. 
If the sequence 
$\left\{\frac{f_{n+1}(p_0)}{f_{n}(p_0)}\mid n\in \bZ_+\right\}$ diverges 
then the support of $\nu$ coincides with $\Xi_{\widetilde\phi}$. 
\end{enumerate}
\end{proposition}

\begin{figure}[!htb]
\centerline{\hbox{\epsfysize=5.0cm\epsfbox{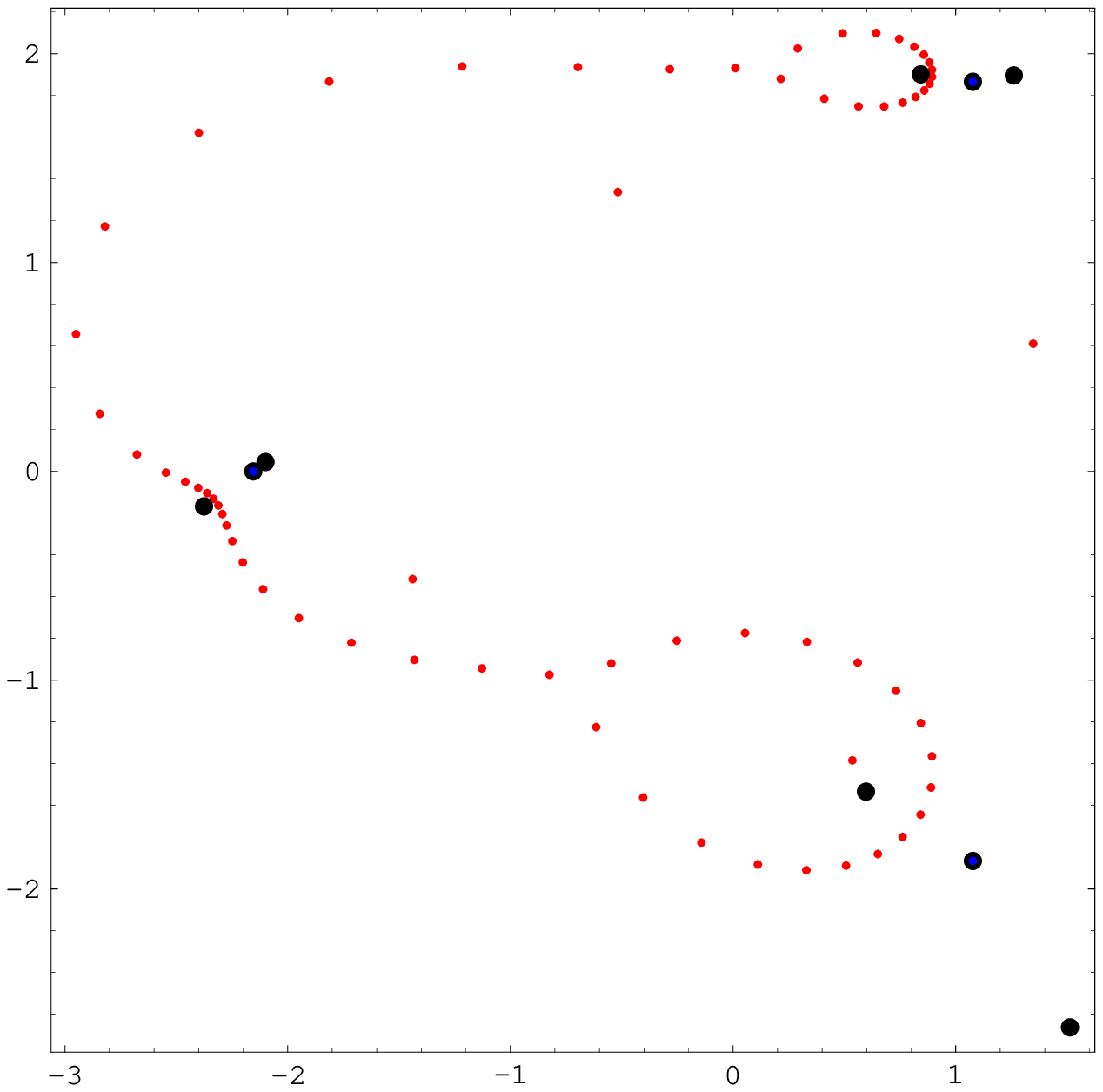}}
\hspace{0.8cm}\hbox{\epsfysize=5.0cm\epsfbox{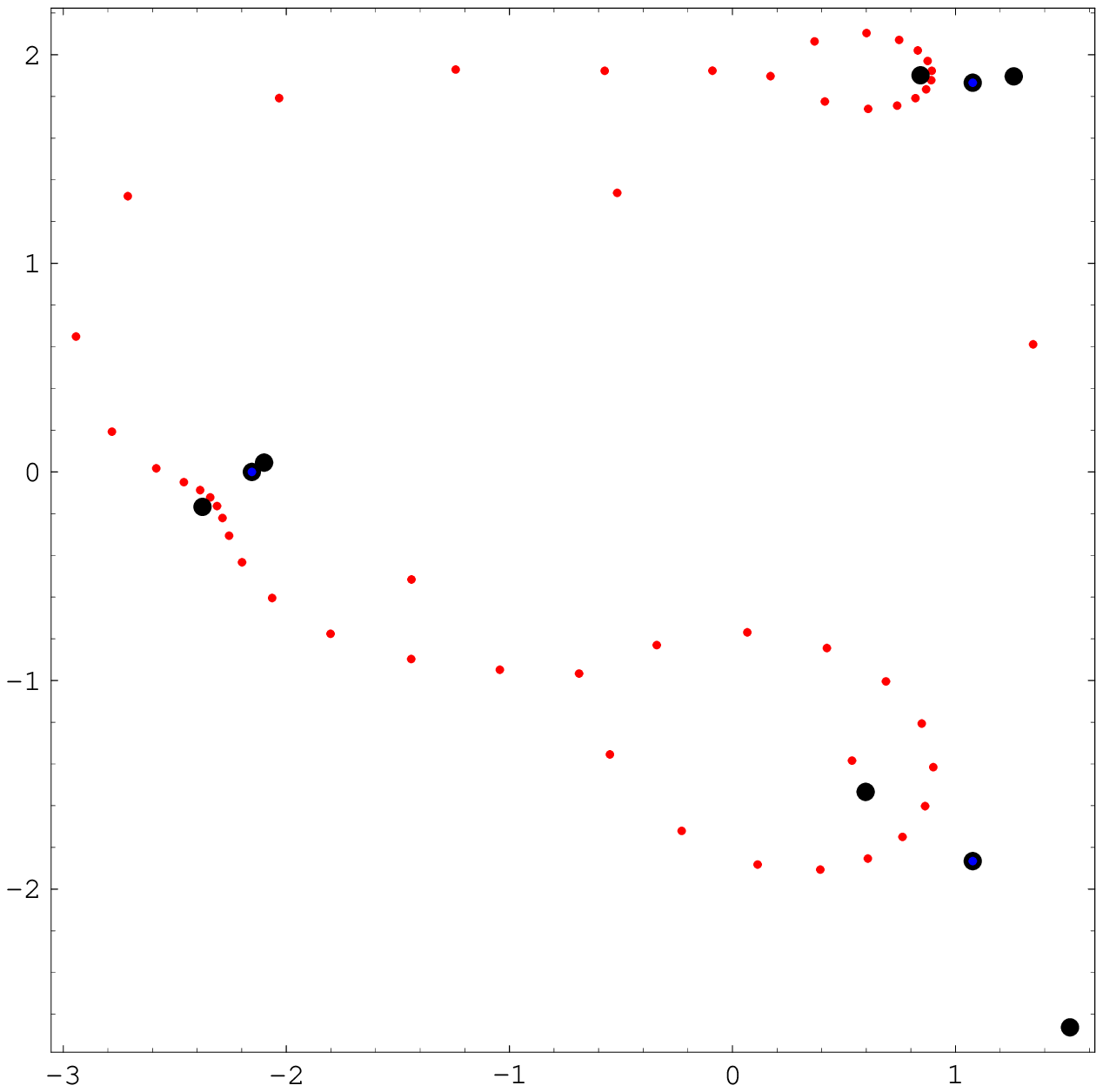}}}
\centerline{\hbox{\epsfysize=5.0cm\epsfbox{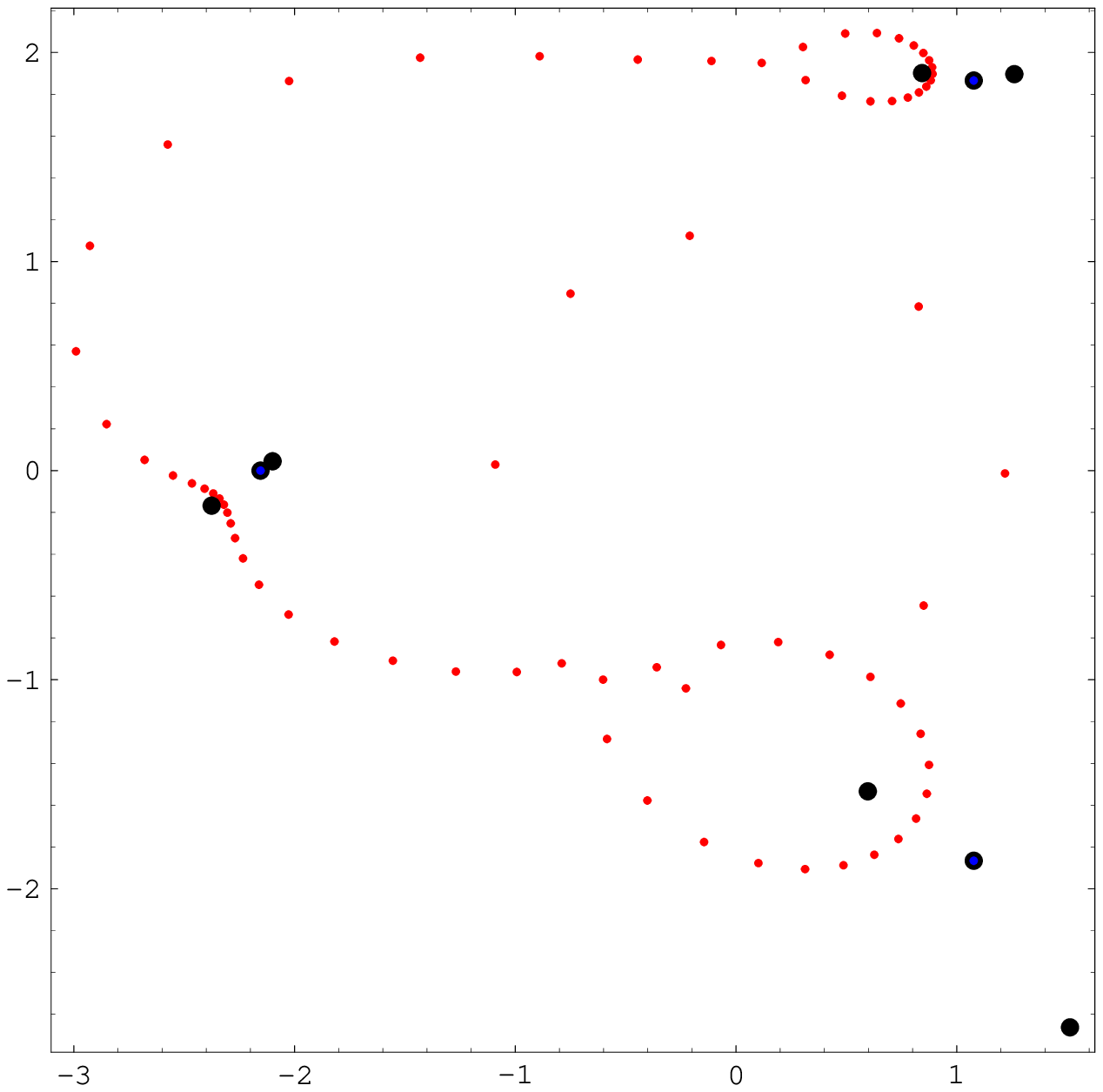}}}
\caption{Zeros of polynomials satisfying the 4-term recurrence
relation $p_{n+1}(z)=p_{n}(z)+(z+1)(z-i)p_{n-1}(z)+(z^3+10)p_{n-2}(z)$}
\label{fig1}
\end{figure}

Let us illustrate the latter result in a concrete situation. Consider a sequence of polynomials $\{p_n(z)\}$ satisfying the 4-term recurrence relation $$p_{n+1}(z)=p_{n}(z)+(z+1)(z-i)p_{n-1}(z)+(z^3+10)p_{n-2}(z)$$ 
with fixed coefficients and starting with some initial triple of polynomials $I=(p_0(z),p_1(z),p_2(z))$. Consider the sequence $\{\bZ(p_n)\}$ of the zero loci of $p_n(z)$. Then, one can roughly divide the zeros in $\bZ(p_n)$ in 2 parts. The zeros in the first part fill  when $n\to \infty$ the maxmod-discriminant  $\Xi_{\widetilde \phi}\subset \bC$ which is a continuous curve. The second part consisting of finitely many points depends of the initial triple and represents the set $\Si_I$ of slowly growing conditions.

\begin{expl1}
 The two
upper pictures show the
zeros of $p_{64}(z)$ and $p_{45}(z)$ for the same initial triple
$p_{0}(z)=0,\;p_{1}(z)=z^4-5i,\;p_{2}(z)=z$. The lower picture shows
the zeros of $p_{64}(z)$ for the same recurrence relation but with another 
initial triple $p_{0}(z)=z^{8}-z^5+i,\;
p_{1}(z)=z-5i,\;p_{2}(z)=5iz^2+z-10$.  One observes that on all three pictures
the zeros split into $2$ parts where the first part  forms a pattern close to a smooth curve $\Xi_{\widetilde\phi}$ 
and a the second part consists of a 
number of isolated points. On the upper pictures there are four
isolated points which practically coincide on both pictures
although the polynomials themselves are coprime. On the lower picture there
are seven isolated points which also form a very stable set, as these points
hardly change if one takes different polynomials $p_{n}(z)$ with the same
initial triple. The nine fat points in these three pictures are the branching points of
the symbol equation $\Psi^3=\Psi^2+(z+1)(z-i)\Psi+(z^3+10).$
\end{expl1}

\begin{remark} 
The isolated points on Figure \ref{fig1} have a 
strong resemblance with the spurious poles that were considered in a 
substantial number of papers on Pad\'e approximation, see e.g.~\cite{Sta} 
and references therein. The study of the exact relation 
between these two objects  can be found  in \cite{BSh}.  
\end{remark}


This paper is organized as follows. In \S 2 we  prove our parametric  Poincar\'e-Perron theorem and   settle the remaining results in \S 3. In \S 4 we consider concrete examples of $3$-term recurrence
relations with polynomial coefficients related to the theory of biorthogonal
polynomials~\cite{IM}.  In \S 5
we discuss a number of related topics and open problems. Finally, in appendix we study the topological structure of the
standard maxmod-discriminant $\Xi_{k}\subset Pol_{k}$. 

\medskip 
\noindent
{\em Acknowledgements.} The authors want to thank the anonymous referee for the constructive criticism  which allowed us to substantially  improve the quality and the clarity of our exposition.  


 \section{Proving  parametric  Poincar\'e-Perron theorem.}

 For $\F=\R,\C$ denote by $\F^k,M_k(\F),\gl_k(\F)$ the
 $k$-dimensional  vector space, the algebra of $(k\times k)$-matrices
 and the subgroup of $k\times k$ invertible matrices over the
 field $\F$.  Denote by $||\cdot ||$ any vector norm on $\F^k$ or on $M_k(\F)$.
 Let $||\cdot ||_2$ be
 the $\ell_2$-norm on $\F^k$ induced by the standard
 inner product $\lan \x,\y\ran:=\y^*\x$ on $\F^k$ and denote by $||\cdot ||_2$ the induced
 operator norm on $M_k(\F)$.
 For $T\in M_k(\C)$ denote by $\rho(T)$ the spectral radius of $T$, i.e. the maximal modulus of the
 eigenvalues of $T$.
 As in \cite{Fr2}, the spaces
 $\P\F^k, \P M_k(\F), \P\gl_k(\F)$ are
 obtained by identifying the orbits of the action of
 $\F^*:=\F\backslash\{0\}$ on the nonzero elements of the corresponding sets by multiplication. 
 Then $\P\R^k,\P M_k(\R)$
 and $\P\C^k,\P M_k(\C)$ are compact real and complex manifolds
 respectively.  To keep our notation  standard  we set $\P^{k-1}=\P\C^k$.
 For $\x\in \F^k\backslash\{\0\}, T\in M_k(\F)\backslash\{0\}$  denote by
 $\hat\x, \hat T$ the induced elements in $\P\F^k, \P M_k(\F)$ respectively.
 Furthermore, $\hat T$ can be viewed as a map from $\P\F^k\backslash \widehat{\ker T}$ to $\P\F^k$.
 Denote by $d(\cdot,\cdot): \P^{k-1}\times \P^{k-1}\to [0,\infty)$ the Fubini-Study metric on $\P^{k-1}$, see e.g.  \cite{GH}.

 Consider an iteration scheme
 \begin{equation}\label{itsch}
 \x_n:=T_n\x_{n-1},\quad \x_0\in \F^k,\;T_n\in M_k(\F), n\in\N.
 \end{equation}
 This system is called \emph{convergent} if $\x_n,n\in\N$ is a
 convergent sequence for each $\x_0\in\F^k$.  This is equivalent
 to the convergence of the infinite product $...T_nT_{n-1}...T_2T_1$,
 which is defined as the limit of $T_nT_{n-1}...T_2T_1$ as
 $n\to\infty$.  For the stationary case $T_n=T, n\in \N$ the
 necessary and sufficient conditions for convergency are well known.
Namely, (i) the spectral radius $\rho(T)$ can not exceed $1$; 
 (ii) if $\rho(T)=1$, then $1$ is an eigenvalue of $T$ and all
 its Jordan blocks have size $1$;  (iii) all other eigenvalues $\lambda$ of
 $T$ different from $1$ satisfy
 $|\lambda|<1$.

 In some instances, as e.g. for problems related to  the Lyapunov exponents in dynamical systems, 
  one is interested if the line spanned by the vector $\x_i$ converges
 for all $\x_0\ne 0$ in some homogeneous open Zariski set in $\F^k$.
 If this condition holds we call (\ref{itsch}) \emph{projectively
 convergent}.

 For the stationary case $0\ne T_n=T\in M_k(\C)$ one can easily check 
 that (\ref{itsch}) is projectively convergent if and only if
 among all the eigenvalues $\lambda$ of $T$ satisfying
 $|\lambda|=\rho(T)$, there is exactly one eigenvalue $\lambda_0$
 which has Jordan blocks of the maximal size.  

 The main point of Theorem \ref{anconvprod} is to show that $\hat\w(\x)$ is continuous in $\mathcal{D}$.
 We  use some technique developed in \cite{Fr2}, in particular,  the arguments of the proof of Theorem \ref{th:linop},  \cite[Theorem 1.2, page 258 ]{Fr2},
 to obtain the continuity of $\hat\w(\x)$.
 Namely, let  $\lambda_{max}(\x)$ is the unique simple eigenvalue of $T(\x)$ with the maximal modulus.
 Then $\lambda_{max}(\x)$ is continuous and nonvanishing in $\mathcal{D}$.   To prove Theorem \ref{anconvprod},
  we replace $T(\x)$ and $T_n(\x)$ by $\frac{1}{\lambda_{max}(\x)} T(\x)$ and $\frac{1}{\lambda_{max}(\x)} T_n(\x)$ for each $n$, respectively.
 Thus we can assume that for each $\x\in \mathcal{D}$ the number $\lambda_{max}=1$ is the simple eigenvalue of $T(\x)$ of the maximal modulus.
 Let $\u(\x),\v(\x)$, $\ne \0$ for each $\x\in\mathcal{D}$, be the right and resp. the left eigenvectors of $T(\x)$ corresponding to $1$, i.e. 
 \begin{equation}\label{defuzvz}
 T(\x)\u(\x)=\u(\x), \quad T\t(\x) \v(\x)=\v(\x), \quad \u\t(\x)\v(\x)=1, \quad \x\in\mathcal{D}.
 \end{equation}
 Since $\u(\x),\v(\x)$ can be chosen continuous  in $\mathcal{D}$ it follows that $\hat \u(\x),\hat \v(\x)\in\P^{n-1}$ are continuous  in $\mathcal{D}$ as well.  In what follows we fix a norm $\|\cdot\|$ on $M_k(\C)$  satisfying for any  $T\in M_k(\C)$ the condition $\|T\|=\|T\t\|$ . 
 
 \medskip 
We need two additional Lemmas; the next one generalizes the inequalities given on pages 254-255 of \cite{Fr2}.
 \begin{lemma}\label{z0epsdeltpert}  Let $\{T_n(\x)\}_{n\in\N}$ be a sequence of  complex-valued   invertible
  $(k\times k)$-matrices, with continuous  entries for $\x\in \mathcal{D}\subset\R^d$.  Assume that this sequence converges uniformly,
  on any compact set in $\mathcal{D}$, to a matrix $T(\x)$.
 Suppose furthermore that for each $\x\in\mathcal{D}$,  the number $\lambda_{max}=1$ is a simple eigenvalue of $T(\x)$, and all other eigenvalues
 of $T(\x)$ lie in the open unit disk $|\lambda|<1$.   Then for any fixed $\x_0$ and for each $\varepsilon>0$ there exists $\delta=\delta(\varepsilon)>0$
 and $N=N(\varepsilon),m=m(\varepsilon)\in\N$ such that
 \begin{eqnarray}\label{z0epsdeltpert1}
 \|T_{n+m}(\x)T_{n+m-1}(\x)\ldots T_{n+1}(\x)-\u(\x_0)\v\t(\x_0)\|<\varepsilon,\\
 \textrm{ for each } n\ge N \textrm{ and }|\x-\x_0|\le \delta. \nonumber
 \end{eqnarray}
 \end{lemma}
 \proof  Since $1$ a simple eigenvalue of $T(\x_0)$ and all other eigenvalues of $T(\x_0)$ lie in the open unit disk,
 it follows that $\lim_{m\to\infty} T^m(\x_0)=\u(\x_0)\v\t(\x_0)$.
 Hence, there exists $m=m(\varepsilon)$ so that $\|T^m(\x_0)-\u(\x_0)\v\t(\x_0)\|<\frac{\varepsilon}{2}$.
 Since the product map $(M_k(\C))^m\to M_k(\C)$  sending a $m$-tuple $X_1,...,X_m$ to their product is continuous, there exists $\delta_1>0$,
 such that for $\|X_i-T(\x_0)\|\le \delta_1$ for $i=1,\ldots,m$ the inequality $\|X_1X_2\ldots X_m -T(\x_0)^m\|
 \le \frac{\varepsilon}{2}$
 holds.  Since $T_n(\x)$ converges uniformly to the continuous  $T(\x)$ on any compact subset of $\mathcal{D}$, it follows that
 there exists $\delta=\delta(\varepsilon)>0$ and $N=N(\varepsilon)\in\N$ such that $\|T_n(\x)-T(\x_0)\|\le
 \frac{\varepsilon}{2}$ for $n\ge N$ and $\|\x-\x_0\|\le \delta$.
 \qed

 Next we give the following modification of \cite[Lemma 5.1]{Fr2}.
  \begin{lemma}\label{rank1ma}  Let $E\in M_k(\C)$ be a 
 matrix of rank one with $\rho(E)>0$, i.e. $E=\v\u\t,\u\t\v=1$.
 Set $O_r:=\{\hat\x\in \P\C^k:\;d(\hat\x,\hat\v)\le r\}$
 such that $O_r\cap \widehat{\ker E}=\emptyset$.  Then
 $\hat E :O_r \to \{\hat\v\}$.
 Then there exists $\varepsilon=\varepsilon(r)$ such that the following conditions hold.
 Assume that $B\in M_k(\C)$ satisfies $\|B-E\|\le \varepsilon$.  Then
 \begin{enumerate}
 \item\label{rank1ma1} $d(\hat B \hat \x,\hat\v)\le \frac{r}{2}$ for each $\hat \x\in\O_r$.
 \item\label{rank1ma2} $d(\hat B\hat\x,\hat B\hat\y)\le \frac{1}{2} d(\hat\x,\hat\y)$ for each $\hat\x,\hat\y\in O_r$ .
 \end{enumerate}

 \end{lemma}
 \proof  Clearly $\widehat {E\x}=\hat\v$ if $\u\t\x\ne 0$.  Hence
 $\hat E :O_r \to \{\hat\v\}$.  Recall that for any $B\in M_k(\C)\backslash\{0\}$ the transformation
 $\hat B:\P\C^k\backslash \widehat{\ker B}\to \P\C^k$ is analytic.   Hence there exists $\varepsilon_0>0$
 so that for $\|B-E\|\le \varepsilon_0,$ $B\ne 0,$ and $\widehat{\ker B}\cap O_r=\emptyset$.
 Thus, $\hat B$ is analytic on $O_r$.  One has that for a small $\epsilon\in(0,\epsilon_0)$ any $\hat B$ satisfying
 $\|B-E\|\le \varepsilon$ is a small perturbation of $\hat E$.  Since $\hat E O_r=\{\hat\v\}$ we deduce  condition
 \ref{rank1ma1} for any $\varepsilon\in (0,\varepsilon_1)$ for some $0<\varepsilon_1<\varepsilon_0$.
 Since $d(\hat E\hat \x,\hat E\hat \y)=0$ for any $\hat\x,\hat\y\in O_r$ we deduce  condition \ref{rank1ma2}
 for some $\varepsilon\in(0,\varepsilon_1)$.
 \qed

\medskip 
 \textbf{Proof of Theorem \ref{anconvprod}.}  Recall that \cite[Theorem 1.2]{Fr2} implies that
the map 
 $\widehat {T_n(\x)} \widehat {T_{n-1}(\x)}\ldots \widehat {T_2(\x)} \widehat {T_1(\x)}$,
 viewed as an automorphism of $\P^{k-1}$, converges pointwise to the transformation $\hat \u(\x)\hat \w\t(\x)$ on $\P^{k-1}\backslash H(\x)$. Now fix some $\x_0\in \mathcal{D}$.  Then there exists $\delta_1>0$ such that the eigenvectors $\u(\x),\v(x)\in \C^k$ satisfying
 (\ref{defuzvz}) can be chosen continuously  in $|\x-\x_0|\le \delta_1$. Since $\v\t(\x)\u(x)=1$ it follows that we can choose
 $\delta_2\in (0,\delta_1)$ such that $\v(\x_0)\u(\x)\ne 0$ for $|\x-\x_0|\le \delta_1$.  To show the continuity  of $\hat\w(\x)$
 in $|\x-\x_0|<\delta$ for some $\delta_3\in (0,\delta_2)$ it is enough to show that the continuity of the limit of
 $\widehat {\v\t(\x_0)}\widehat {T_n(\x)} \widehat {T_{n-1}(\x)}\ldots \widehat {T_2(\x)} \widehat {T_1(\x)}$
 in $|\x-\x_0|< \delta_3$.  Since each $T_j(\x)\in \gl(k,\C)$ is continuous, it is enough to show the continuity of the limit of
 $\widehat {\v\t(\x_0)}\widehat {T_{n+N}(\x)} \widehat {T_{n+N-1}(\x)}\ldots \widehat {T_{N+2}(\x)} \widehat {T_{N+1}(\x)}$
 in $|\x-\x_0|< \delta_3$.  One has to show that a subsequence of
  $\widehat {\v\t(\x_0)}\widehat {T_{n+N}(\x)} \widehat {T_{n+N-1}(\x)}\ldots \widehat {T_{N+2}(\x)} \widehat {T_{N+1}(\x)}$
 converges uniformly in $|\x-\x_0|\le \delta_3$ in the Fubini-Study metric $d(\cdot,\cdot)$ on $\P^{k-1}$.
 
 This is done by using Lemmas \ref{z0epsdeltpert}-\ref{rank1ma}. 
Namely,  let $E=\v(\x_0)\u\t(\x_0)$.  Choose $r>0, \varepsilon>0$ so that the conditions \ref{rank1ma1}-\ref{rank1ma2} of Lemma \ref{rank1ma}
 hold.  Next choose $N,\delta$ so  that the condition (\ref{z0epsdeltpert1}) holds.  Let $\delta_3=\min(\delta,\delta_2)$.
 Define
 $$B_j(\x)=T_{(j-1)m+N+1}\t(\x)\ldots T_{jm+N}\t(\x) \quad \textrm{ for } j=1,\ldots.$$

 Assume that  $|\x-\x_0|\le \delta_3$.  We claim that the sequence $\widehat{B_1(\x)}\ldots\widehat{B_l(\x)}\widehat{\v(\x_0)}$ converges uniformly  to $\w(\x)\in O_r$ 
 in the Fubini-Study metric.  First observe that (\ref{z0epsdeltpert1})
 yields that $\widehat{B_j(\x)}O_r\subset O_{\frac{r}{2}}$, by  condition \ref{rank1ma1} of Lemma \ref{rank1ma}.
 Hence $\hat \y_l(\x)=\widehat{B_1(\x)}\ldots\widehat{B_l(\x)}\widehat{\v(\x_0)}\subset O_{\frac{r}{2}}$.
 Condition \ref{rank1ma1} of Lemma \ref{rank1ma} yields that
 \begin{eqnarray*}
 d(\widehat{B_1(\x)}\ldots\widehat{B_l(\x)}\widehat{\v(\x_0)},\widehat{B_1(\x)}\ldots\widehat{B_{l+p}(\x)}\widehat{\v(\x_0)})\le\\
\le \frac{1}{2^l}d(\widehat{\v(\x_0)},\widehat{B_{l+1}(\x)}\ldots\widehat{B_{l+p}(\x)}\widehat{\v(\x_0)})\le \frac{l}{2^l}.
 \end{eqnarray*}
 Hence the sequence $\hat \y_l, l\in\N$ converges uniformly in $|\x-\x_0|\le \delta_3$ to $\hat \y(\x)$.
 Since each $\hat \y_l(\x)$ is continuous in $|\x-\x_0|\le \delta_3$ it follows that $\y(\x)$ is continuous in the open disk
 $|\x-\x_0|<\delta_3$.  \qed

\section{Proving  remaining results}\label{s2}



Consider  a recurrence relation with varying coefficients of the form
\begin{equation}\label{eq:varrec}
	u_{n+k}+\al_{1,n}u_{n+k-1}+\al_{2,n}u_{n+k-2}+\ldots+\al_{k,n}u_{n}=0
\end{equation}
where $\al_{k,n}\neq 0$. Assume that for all $i\in \{1,\ldots,k\}$ the one has 
$\lim_{n\to \infty}\al_{i,n}=:\alpha_{i}$ and denote the
limiting recurrence relation of~\eqref{eq:varrec} by
\begin{equation}\label{eq:limrec}
v_{n+k}+\al_{1}v_{n+k-1}
+\al_{2}v_{n+k-2}+\ldots
+\al_{k}v_{n}=0.
\end{equation}

\begin{lemma}\label{lm:varcoef}

If~\eqref{eq:limrec} is maxmod-generic then
     $\lim_{n\to\infty}\frac {v_{n+1}}{v_{n}}$ exists
     and equals ${ \la_{max}}$  on the complement of a complex
     hyperplane $H\subset \bC^k$, where $ \la_{max}$
denotes the leading spectral number of \eqref{eq:limrec}. 
\end{lemma}

\begin{proof}
The result is an immediate corollary of Theorem~\ref{th:linop} 
applied to the family of linear operators $\{T_{n}\mid n\in \bN\}$ whose
action on $\bC^{k+1}$ is given by
$$T_{n}=\begin{pmatrix} 1& \al_{1,n} & \al_{2,n} & \al_{3,n}&
\ldots & \al_{k,n} \\
1 & 0 &0 &0& \ldots & 0\\
0 & 1 & 0 &0& \ldots &0 \\
\vdots & \vdots & \vdots & \vdots& \ddots &\vdots \\
  0 &0 &0& \ldots &1 &0 \end{pmatrix},\quad n\in \bN.$$ 
  Notice that $\al_{k,n}\neq 0$. \qed
\end{proof}

The exceptional complex hyperplane $H$  is called   the {\em hyperplane of slow growth}.
From Theorem~\ref{th:linop} we deduce the following corollary
(cf.~\cite[Proposition 2.1]{Fr1}):

\begin{corollary}\label{cor:linop}
Under the assumptions of Theorem~\ref{th:linop} there exists a sequence of
real numbers $\{\theta_n\}_{n \in \bN}$ such that
$$\lim_{n\to\infty}e^{i\theta_n}\frac{T_{n}T_{n-1}\cdots T_{1}}
{||T_{n}T_{n-1}\cdots T_{1}||}={\bf u}_{max}{\bf w}^{t},$$
the convergence taking place in $\text{{\rm M}}_{k}(\bC)$.
\end{corollary}

We now settle Theorem~\ref{th:Exist}. 

\medskip 
\begin{proof}[Theorem~\ref{th:Exist} (i)]
 Let us first 
determine $\Sigma_{I}$ and at the same time prove the  
pointwise convergence of $\frac 
{f_{n+1}(x_{1},\ldots,x_{d})}{f_{n}(x_{1},\ldots,x_{d})}$ to 
$\Psi_{max}(x_{1},\ldots,x_{d})$ in the complement $\Omega\setminus 
(\Xi_{\widetilde \phi}\cup \Sigma_{I})$. 
One can view  the set-up of
Theorem~\ref{th:Exist} as the situation when the recurrence
relation~\eqref{eq:varrec} depends on additional parameters
$\x=(x_{1},\ldots,x_{d})$. Thus if at a given
point $p\in \Omega\subseteq \bR^d$ the limit $\lim_{n\to\infty}\frac
{f_{n+1}(p)}{f_{n}(p)}$ does not exist then  $p$ lies in the induced maxmod-discriminant
$\Xi_{\widetilde \phi}$ (cf.~Definition~\ref{def7}). On the other hand, if  $\lim_{n\to\infty}\frac
{f_{n+1}(p)}{f_{n}(p)}$ exists but is not equal to $\Psi_{max}(p)$ and, additionally,    
 $p$ lies in $\Omega\setminus \Xi_{\widetilde \phi}$ then the corresponding initial $k$-tuple
$I(p)=\big(f_{0}(p),\ldots,f_{k-1}(p)\big)$
belongs to  the hyperplane $H(p)$ of slow growth   at the given
point $p$, see Lemma~\ref{lm:varcoef}.
 The latter set of points $p$ 
is by definition the  set  $\Si_{I}$ of slowly growing initial conditions.  
For varying coefficients the hyperplane of slow growth at a
given point $p\in \Omega\setminus \Xi_{\widetilde\phi}$ is determined by
$\lim_{n\to \infty}e^{i\theta_n}\frac{T_{n}T_{n-1}\cdots T_{1}}
{\vert \vert T_{n}T_{n-1}\cdots T_{1}\vert\vert}$, see
Corollary~\ref{cor:linop}. Indeed, Theorem~\ref{th:linop} implies that
in this case the hyperplane of slow growth $H(p)$  at $p$
consists of all vectors ${\bf x}(p)\in \bC^{k}$ such that
${\bf w}^{t}(p){\bf x}(p)=0$. 
\end{proof}

To settle  Theorem~\ref{th:Exist} (ii) we  assume 
that $\overline \phi_n\rightrightarrows \widetilde \phi=
\big(\tilde \phi_1,\tilde \phi_2,\ldots,\tilde \phi_k\big)$ and all considered 
functions are continuous in $\Omega$. 
Let $p_0\in \Omega\setminus (\Xi_{\widetilde \phi}\cup \Sigma_{I})$. From 
Corollary~\ref{cor:linop} and the arguments of 
\cite[\S 4 and \S 5]{Fr1} it follows that one can find a sufficiently small 
neighborhood $\calO_{p_0}\subset \Omega\setminus (\Xi_{\widetilde \phi}\cup 
\Sigma_{I})$ of $p_0$ 
such that for any $\eps>0$ there exists $N_{\eps}\in \bN$ satisfying
\begin{equation}\label{fake}
\left\vert\left\vert e^{i\theta_n(p)}\frac{K_n(p)}
{\vert \vert K_n(p)\vert\vert}-
{\bf u}_{max}(p){\bf w}^{t}(p)\right\vert\right\vert\le \eps\,\text{ for }\,
p\in \calO_{p_0},\,n\ge N_{\eps},
\end{equation}
where $K_n(p):=T_{n}(p)T_{n-1}(p)\cdots T_{1}(p)$. Clearly, 
${\bf u}_{max}(p){\bf w}^{t}(p)$ is a rank one matrix and so it has 
precisely one simple eigenvalue of maximum modulus for all $p\in \calO_{p_0}$. 
From~\eqref{fake} we deduce that there exists $N\in \bN$ such that the 
polynomial $\text{det}(\la I_k-K_n(p))$ has precisely one simple eigenvalue 
$\la_{max,n}(p)$ of maximum modulus for all $p\in \calO_{p_0}$ and $n\ge N$. The 
implicit function theorem implies that $\la_{max,n}(p)$ is continuous  in $\calO_{p_0}$
for all $n\ge N$, so that $\left\{\dfrac{K_n(p)}{\la_{max,n}(p)}\right\}_{n\ge N}$
is a sequence of matrix-valued functions which are continuous  in $\calO_{p_0}$. 
On the other hand, it is not difficult to see that this sequence converges 
pointwise in $\calO_{p_0}$ to ${\bf u}_{max}(p){\bf w}^{t}(p)$ (see e.g.~the proof 
of Theorem~\ref{th:linop} given in~\cite{Fr1}). Together with~\eqref{fake} this
implies that 
$\left\{\dfrac{\vert\vert K_n(p)\vert\vert}{|\la_{max,n}(p)|}\right\}_{n\ge N}$ 
is a sequence of (continuous) functions that converges pointwise to the 
constant function 1 on $\calO_{p_0}$. It follows that there exists some open
subset $\calO_{p_0}'\subset \calO_{p_0}$ with $p_0\in \calO_{p_0}'$ such that
$\left\{\dfrac{\vert\vert K_n(p)\vert\vert}{|\la_{max,n}(p)|}\right\}_{n\ge N}$ 
is a uniformly bounded sequence of functions on $\calO_{p_0}'$. Thus 
$\left\{\dfrac{K_n(p)}{\la_{max,n}(p)}\right\}_{n\ge N}$ is a bounded 
sequence of continuous matrix-valued functions which converges pointwise to 
the function ${\bf u}_{max}(p){\bf w}^{t}(p)$ in $\calO_{p_0}'$. Invoking
Vitali's theorem we conclude that
\begin{equation}\label{f2}
\frac{K_n(p)}{\la_{max,n}(p)}\rightrightarrows {\bf u}_{max}(p){\bf w}^{t}(p)
\,\text{ in }\,\calO_{p_0}'.
\end{equation} \qed


\begin{notation}\label{not2} 
In notation of Theorem~\ref{th:Exist} 
let $V=\Omega\times \bC^k$ denote the Cartesian product of $\Omega$ and the
 linear space $\bC^k$ of all initial $k$-tuples at a point in $\Omega$. Any initial $k$-tuple
$I=\big(f_{0}(\x),\ldots,    f_{k-1}(\x)\big)$ of smooth (respectively, analytic) 
functions in    $\Omega$ can be considered as a smooth (respectively, analytic) section of 
$V$ viewed as  a trivial vector bundle over $\Omega$.  Denote by $\widetilde V$ the
   restriction of the bundle $V$ to the subset $\Omega\setminus \Xi_{\widetilde \phi}$
   of the base.
\end{notation}

\begin{definition}
Let $\Delta_{\phi}$ denote the subset of $\widetilde V$ consisting of
   all pairs $(p,I(p))$ such that the asymptotic ratio of the
recurrence~\eqref{eq:General} evaluated at $p$ with initial $k$-tuple
$I(p)\in \bC^k$
does not coincide with the leading root of~\eqref{eq:Symb} at $p$. The
set $\Delta_{\phi}$ is called the {\em subvariety of slow growth}.
\end{definition}

\begin{corollary}\label{pr:1}
In the above notation the subvariety $\Delta_{\phi}$ of slow growth is a smooth (resp. an
 analytic)  subbundle of complex hyperplanes
in $\widetilde V$ provided that the family of $k$-tuples
$$\left\{\overline \phi_n=\big(\phi_{1,n}(\x),\phi_{2,n}(\x), \ldots,\phi_{k,n}(\x)\big)\mid
n\in \bZ_+\right\}$$
with $\phi_{k,n}(\x)$ non-vanishing in $\Omega$ 
and its limit $k$-tuple $\widetilde \phi=\Big(\widetilde \phi_{1}(\x),
\widetilde \phi_{2}(\x),\ldots,
\widetilde \phi_{k}(\x)\Big)$ consist of smooth (resp. 
analytic)  functions that satisfy 
$\overline \phi_{n}\rightrightarrows 
\widetilde \phi$ 
in $\Omega$ in the corresponding category. 
\end{corollary}

\begin{proof} Indeed, by Theorem ~\ref{th:Exist}  (ii) the  family of hyperplanes of slow growth depend  smoothly resp. analytically on $p\in \Omega\setminus \Xi_{\widetilde \phi}$ in the corresponding category. \qed  
\end{proof} 

\begin{definition}
Let $\xi: E\to B$ be a (complex) vector bundle over a base $B$. A smooth (respectively, analytic) section $S:B\to E$ is called 
{\em transversal} to a given 
smooth (respectively, analytic) submanifold $\mathcal H \subset E$ if  at 
each point $p$ of the
intersection  $S\cap \mathcal H$ the sum of the tangent spaces at $p$ to
$S$ and $\mathcal H$ coincides
with the tangent space  at $p$ to the ambient space $E$. A subset $X$ of a  topological space $Y$  is called {\em massive} if $X$ can be represented as  the
intersection of at most countably many open dense subsets in $Y$. 
\end{definition}

\begin{remark}\label{rem4}
Thom's transversality theorem, see e.g.~\cite{GG}, implies that the set of
all smooth sections of the bundle $ \xi: E\to B$ which are transversal to a given smooth
subbundle
$\mathcal H\subset E$ is a massive subset of the set of all sections. The same 
holds in the analytic category provided that the fibration is trivial.
\end{remark}

\begin{lemma}\label{cor:1}
If an initial $k$-tuple $I=\big(f_{0}(\x),\ldots,    f_{k-1}(\x)\big)$ of smooth (respectively, analytic) functions 
in $\Omega$  is transversal to the subvariety $\Delta_{\phi}$ of slow growth  then the
set $\Si_{I}$ of slowly growing initial conditions is
  either empty or a smooth (respectively, analytic) subvariety in $\Omega$ of 
real codimension two.
\end{lemma}

\begin{proof}
By Thom's transversality theorem the set of all smooth 
initial $k$-tuples which are transversal to $\Delta_{\Phi}$  is a
massive subset of the set of all possible smooth initial
$k$-tuples (cf.~Remark~\ref{rem4}). Combined with Proposition~\ref{pr:1}
this implies, in particular, that for a given initial $k$-tuple $I$
transversal to $\Delta_{\phi}$ one has that $I$ (considered as a
section of $\widetilde V$) and $\Delta_{\Phi}$ are smooth submanifolds
of $\widetilde V$ transversal to each other, where $\widetilde V$ is the
restriction of the bundle $V=\Omega\times \bC^k$ to the subset
$\Omega\setminus \Xi_{\widetilde \phi}$ of the base (cf.~Notation~\ref{not2}). This
transversality property implies  that the intersection
$I\cap \Delta_{\phi}$ is either empty or a smooth submanifold of
$\widetilde V$ of real codimension $2$. Notice that the image
$\xi(I\cap \Delta_{\phi})$ of the
projection $\xi:\widetilde V\to \Omega\setminus \Xi_{\widetilde \phi}$
to the base is exactly $\Sigma_{I}$. Actually, $\xi$
induces a diffeomorphism between $I\cap \Delta_{\phi}$ and
$\Sigma_{I}$. To see this one simply notes that $\xi$ maps the
section $I$ diffeomorphically to
$\Omega\setminus \Xi_{\widetilde \phi}$ and that the intersection
$I\cap \Delta_{\phi}$ is a smooth submanifold of the section $I$. The proof 
of the corresponding result in the analytic category is analogous. \qed
\end{proof}

In particular, in the one-dimensional case one has the following. 

\begin{lemma}\label{cor:2}
 If the coefficients
$\big(\phi_{1,n}(z),\ldots, \phi_{k,n}(z)\big)$, $n\in \bZ_+$, and the initial
$k$-tuple $I=\big(f_{0}(z),\ldots,f_{k-1}(z)\big)$ of
recurrence~\eqref{eq:General} are complex analytic functions in an
open set $\Omega\subseteq \bC$, and $\phi_{k,n}(z)\neq 0$ in $\Omega$  then either $\,\Xi_{\widetilde \phi}=\Omega$ or 
$\,\Xi_{\widetilde \phi}$ is a union of real analytic curves while $\Si_{I}$ is
either the whole $\Omega$ or consists of isolated points.
\end{lemma}

\begin{proof} 
By Corollary~\ref{pr:1} the set $\Delta_{\Phi}$ is analytic 
in the analytic category. Thus in this case the asymptotic symbol
equation (\ref{eq:Symb}) is  either maxmod-nongeneric everywhere in
$\Omega\subseteq \bC$ or is  maxmod-nongeneric on a one-dimensional real analytic subset, which
proves the statement in part (i) concerning $\Xi_{\widetilde \phi}$. Note that, in fact,  in
the analytic  category $\Xi_{\widetilde \phi}$ is always a real semianalytic set and  $\Sigma_{I}$ is analytic. 
Therefore,  $\Sigma_{I}$
is either a complex analytic curve, in which case it
coincides with $\Omega$, or an analytic 
zero-dimensional subset of $\Omega$, i.e., the union of isolated points.\qed
\end{proof}

We finally turn to Proposition~\ref{pr:2}. To settle it we need an additional lemma. 
which  a simple consequence of Theorem~\ref{th:ordrec}.

\begin{notation}
  Let $Rec_{k}$ be the $k$-dimensional complex linear space    consisting of all $(k+1)$-term recurrence relations with constant
coefficients of the form~\eqref{eq:Basic}. As before we denote by  $\bC^k$ the $k$-dimensional
complex linear space  of all initial $k$-tuples $(u_{0},...,u_{k-1})$.
\end{notation}

\begin{definition}\label{domin}
A maxmod-nongeneric recurrence relation in $Rec_{k}$ with initial $k$-tuple
$in_k\in \bC^k$ is said to be of {\em dominant type} if the following conditions
are satisfied. Let $\la_1,\ldots,\la_r$, $r\le k$, denote all distinct
spectral numbers with maximal absolute value and assume that these have
multiplicities $m_1,\ldots,m_r$, respectively. Then there exists a unique
index $i_0\in\{1,\ldots,r\}$ such that $m_i<m_{i_0}$ for
$i\in \{1,\ldots,r\}\setminus \{i_{0}\}$ and the initial $k$-tuple $in_k$ is
{\em fast growing} in the sense that the degree of the polynomial $P_{i_0}$
in~\eqref{eq:leadasymp} corresponding to $\la_{i_0}$ is precisely
$m_{i_0}-1$. The number $\la_{i_0}$ is called the {\em dominant
spectral number} of this recurrence relation.
\end{definition}

\begin{lemma}\label{lm:basrec}
In the above notation the following is true:
\begin{enumerate}
\item[(i)] The set of all slowly growing initial $k$-tuples
with respect to a given maxmod-generic recurrence relation in $Rec_{k}$ is a complex
hyperplane $\SG_k$ in $\bC^k$. The set $\SG_k$ is called the {\em hyperplane of slow growth}.
\item[(ii)] For any maxmod-generic recurrence relation in $Rec_{k}$ and any
fast growing initial $k$-tuple $(u_{0},\ldots,u_{k-1})$ the limit
$\lim_{n\to\infty}\frac{u_{n+1}}{u_{n}}$ exists and     coincides with the leading spectral number $\la_{max}$, that is,
     the (unique) root of the characteristic equation~\eqref{eq:Char} with     maximal absolute value.
\item[(iii)] Given a maxmod-nongeneric recurrence relation of
dominant type in $Rec_{k}$ the limit $\lim_{n\to\infty}\frac{u_{n+1}}{u_{n}}$ exists and
coincides with the dominant spectral number.
\item[(iv)] For any maxmod-nongeneric recurrence relation of
nondominant type in $Rec_{k}$ the set of initial $k$-tuples for which $\lim_{n\to\infty}\frac{u_{n+1}}{u_{n}}$
exists is a union of complex subspaces of $\bC^k$ of positive codimensions.
This union is called the {\em exceptional variety}.
\end{enumerate}
\end{lemma}

\begin{proof}  
  In order to prove (i) notice that the coefficient $\kappa_{max}$ in
  Definition~\ref{df:slow} is a nontrivial linear combination of the entries of the initial $k$-tuple $(u_{0},\ldots,u_{k-1})$ with
coefficients depending   on $\al_{1},\ldots,\al_{k}$. Therefore, the condition
  $\kappa_{max}=0$ determines a complex hyperplane $\SG_k$ in
  $\bC^k$. One can easily see that the hyperplane of slow growth is the direct sum of all Jordan blocks
  corresponding to the spectral numbers of a given recurrence~\eqref{eq:Basic} other than the leading one.

The assumptions of part  (ii) together with~\eqref{eq:leadasymp} yield
  $u_{n}=\kappa_{max}\la_{max}^{n}+\ldots$ for $n\in \bZ_+$, where the dots
stand for the  remaining terms in (\ref{eq:leadasymp}) corresponding to the
 spectral numbers whose absolute values are strictly smaller than $|\la_{max}|$.
Therefore, the quotient $\frac{u_{n+1}}{u_{n}}$ has a limit as $n\to\infty$
and this limit coincides with $\la_{max}$, as required. By
  definition $\la_{max}$ is a root of~\eqref{eq:Char}, which completes
  the proof of (ii). This last step can alternatively be carried out by
   dividing both sides of~\eqref{eq:Basic} by $u_{n-k+1}$ and then letting
$n\to \infty$. In view of Definition~\ref{domin} the same arguments show that the assertion in (iii) is true as well.

For the proof of part (iv) we proceed as follows. Take any maxmod-nongeneric recurrence relation of the form~\eqref{eq:Basic} and let
$\la_{1},\ldots,\la_{r}$, $r\le k$, be all its distinct spectral numbers
with maximal absolute value. Thus $|\la_i|=|\la_{max}|$ if and only $1\le i\le r$. Choose an initial
   $k$-tuple $IT=(u_{0},\ldots,u_{k-1})$ and denote by $P_{1},...,P_{r}$ the polynomials in~\eqref{eq:leadasymp} corresponding to
$\la_{1},\ldots,\la_{r}$, respectively,  for the sequence $\{u_{n}\mid n\in \bZ_+\}$ constructed using the given
recurrence with initial $k$-tuple $IT$ as above. Assuming, as we may, that
our recurrence relation is nontrivial we get from~\eqref{eq:leadasymp} that
$\la_i\neq0$ if $1\le i\le r$. We may further assume that the degrees
$d_1,\ldots,d_r$ of the polynomials $P_{1},...,P_{r}$, respectively, satisfy $d_1\ge\ldots\ge d_r$.
Under these conditions we now prove that
$\lim_{n\to \infty}\frac{u_{n+1}}{u_{n}}$ exists if and only if exactly
one of the polynomials $P_{1},\ldots,P_{r}$
is nonvanishing.
A direct check analogous to the proof of part (ii)
  shows that if only the polynomial $P_{1}$ is nonvanishing then
$\lim_{n\to \infty}\frac{u_{n+1}}{u_{n}}=\la_{1}$. If $r\ge 2$ and
$s\in \{2,\ldots,r\}$ is such that $P_{1},\ldots,P_{s}$ are all nonvanishing
polynomials among $P_{1},\ldots,P_{r}$ then using again~\eqref{eq:leadasymp} we get
   $$\frac {u_{n+1}}{u_{n}}=\frac {P_{1}(n+1)+P_{2}(n+1)\left(\dfrac
   {\la_{2}}{\la_{1}}\right)^{n+1}+\ldots +P_{s}(n+1)\left(\dfrac
   {\la_{s}}{\la_{1}}\right)^{n+1}+o(1)}{P_{1}(n)+P_{2}(n)\left(\dfrac
   {\la_{2}}{\la_{1}}\right)^n+\ldots +P_{s}(n)\left(\dfrac
   {\la_{s}}{\la_{1}}\right)^n+o(1)}.$$
Since $\big|\frac{\la_{i}}{\la_{1}}\big|=1$ and $\la_i\neq \la_1$,
$2\le i\le s$, it follows that if $d_1=d_2$ then the expression in the
right-hand side has no limit as $n\to\infty$. Therefore, if such a limit
exists then $d_1>d_2$, which gives us a complex subspace of $\bC^k$ of
(positive) codimension equal to $d_1-d_2$. Thus the exceptional variety is a
union of complex subspaces of $\bC^k$ of (in general) different codimensions. \qed
\end{proof}

\medskip 
Now we  prove Proposition~\ref{pr:2}. 
\medskip 

\begin{proof} Under the assumptions of the 
Proposition~\ref{pr:2} the leading root of the
asymptotic symbol equation~\eqref{eq:Symb} is a  well-defined analytic 
function in a sufficiently small neighborhood of $\Sigma_{I}$.
Therefore, the residue distribution~\eqref{resid} associated to this leading
root vanishes in a neighborhood of $\Sigma_{I}$.
Thus the set $\Sigma_{I}$ of slowly growing initial conditions can be
deleted from the support of the asymptotic ratio distribution $\nu$, which 
proves (i).

In order to prove that under the nondegeneracy assumptions of (ii)
the support of $\nu$ coincides with $\Xi_{\widetilde \phi}$ we show that in this case 
the sequence $\frac{f_{n+1}(p)}{f_{n}(p)}$ diverges for almost all 
$p\in \Xi_{\widetilde \phi}$. Indeed, due to analyticity  the first condition of 
(ii) implies that the recurrence relation is of nondominant type almost 
everywhere in $\Xi_{\widetilde \phi}$. The second condition of (ii) then implies that 
the sequence $\frac {f_{n+1}(p)}{f_{n}(p)}$ diverges almost everywhere in 
$\Xi_{\widetilde \phi}$, which settles part (ii) of the proposition. \qed
\end{proof}

\section{Application to biorthogonal polynomials }\label{s3}

Here we calculate explicitly the induced maxmod-discriminant  $\Xi_{\widetilde \phi}$ and the asymptotic ratio distribution $\nu$ in a number of cases. 
We illustrate the relation of  $\Xi_{\widetilde \phi}$ and $\Si_I$ with the roots of polynomials satisfying certain recurrence relations. 
 Let us first recall some definitions from~\cite{IM}. In {\em op.~cit.~}the
authors introduced two types of polynomial families related to what they call
$R_{I}$- and $R_{II}$-type continued fractions, respectively. These families
were later studied in \cite {Zh}.

A polynomial family of type $R_{I}$ is a system of monic
polynomials generated by
$$p_{n+1}(z)=(z-c_{n+1})p_{n}(z)-\la_{n+1}(z-a_{n+1})p_{n-1}(z)$$
with $p_{-1}(z)=0$, $p_{0}(z)=1$ and $p_{n}(a_{n+1})\la_{n+1}\neq 0$ for
$n\in \bZ_+$.

A polynomial family of type $R_{II}$ is a system of monic
polynomials generated by
$$p_{n+1}(z)=(z-c_{n+1})p_{n}(z)-\la_{n+1}(z-a_{n+1})(z-b_{n+1})p_{n-1}(z)$$
with $p_{-1}(z)=0$, $p_{0}(z)=1$ and
$p_{n}(a_{n+1})p_{n}(b_{n+1})\la_{n+1}\neq 0$ for $n\in \bZ_+$.

Assuming that
$$\lim_{n\to \infty}a_{n}=A,\;\lim_{n\to \infty}b_{n}=B,
\;\lim_{n\to \infty}c_{n}=C\text{ and }\lim_{n\to \infty}\la_{n}=\La,$$
where $A,B,C,\La$ are complex  numbers, we
describe the asymptotic ratio $\Psi_{max}(z)=\lim_{n\to\infty}\frac
{p_{n+1}(z)}{p_{n}(z)}$ and the asymptotic ratio measure $\nu$ for
$R_{I}$- and $R_{II}$-type polynomial families.
Using our previous results we restrict ourselves to 
a recurrence relation of the form
\begin{equation}\label{genform}
p_{n+1}(z)=Q_{1}(z)p_{n}(z)+Q_{2}(z)p_{n-1}(z),
\end{equation}
with the initial triple $p_{-1}(z)=0$,\ $p_{0}(z)=$1, $\deg Q_{1}(z)=1$ and
$\deg Q_{2}(z)\le 2$. 
By Theorem~\ref{th:Exist} and Corollary~\ref{cor:2}
the asymptotic ratio $\Psi_{max}(z)$ satisfies the
symbol equation
$$\Psi^2_{max}(z)=Q_{1}(z)\Psi_{max}(z)+Q_{2}(z)$$
everywhere in the complement of the union of the induced maxmod-discriminant 
$\Xi_{\widetilde \phi}$ and a finite set of (isolated) points. 
Notice also that the asymptotic ratio distribution $\nu$ satisfies the relation
$$\nu=\frac{\partial\Psi_{max}(z)}{\partial \bar z},$$
where the r.h.s. is interpreted as a distribution.  
By our previous results the support of $\nu$ coincides with $\Xi_{\widetilde \phi}$. In order to
describe $\Xi_{\widetilde \phi}$ we need the following simple lemma.

\begin{lemma}\label{lm:deg2}
A quadratic polynomial $t^2+a_{1}t+a_{2}$ has two roots with
the same absolute value if and only if there exists a real number $\eps\in
[1,\infty)$ such that $\eps a^2_{1} - 4 a_{2}=0$.
\end{lemma}

\begin{proof}
The roots of the polynomial are given by
$t_{1,2}=\frac{-a_{1}\pm\sqrt{a^2_{1}-4 a_{2}}}{2}$, so that if
$\vert t_{1}\vert=\vert t_{2}\vert$ then the complex numbers $a_{1}$ and
$\sqrt{a^2_{1}-4 a_{2}}$ considered as vectors in $\bC$ must be orthogonal.
Thus
$\big\vert\text{Arg}(a_{1})-\text{Arg}{\sqrt{a^2_{1}-4 a_{2}}}\big\vert
=\frac{\pi}{2}$, which is equivalent to saying that
$\frac{\sqrt{a^2_{1}-4a_{2}}}{a_{1}}$ is purely imaginary and therefore
$\frac{a^2_{1}-4 a_{2}}{a^2_{1}}\in (-\infty,0]$. The converse statement is
obvious. \qed
\end{proof}

\begin{corollary}\label{cor3}
The induced maxmod-discriminant $\Xi_{\widetilde \phi}\subset \bC$ is the set of
all solutions $z$ to the equation
\begin{equation}\label{eq:disc}
\eps Q_{1}^2(z)+4Q_{2}(z)=0,
\end{equation}
where $\eps\in [1,\infty)$.
\end{corollary}

We start with  the situation when $Q_{1}(z)$ and $Q_{2}(z)$
are real polynomials.  Let $C$
denote the unique (real) root of $Q_{1}(z)$ and denote by $D, E$ the branching
points of the symbol equation, i.e., the roots of
$Q_{1}^2(z)+4Q_{2}(z)=0$. If $Q_{1}(z)$ and $Q_{2}(z)$ are real there
are three basic cases:
\begin{enumerate}
\item $D$ and $E$ are real, $D<E$ and $C\in [D,E]$;
\item $D$ and $E$ are complex conjugate;
\item $D$ and $E$ are real, $D<E$ and $C\notin [D,E]$.
\end{enumerate}

\begin{case}
The discriminant $\Xi_{\widetilde \phi}$ is  the interval $[D,E]$. The density $\rho_\nu$ of
$\nu$ in this case is
positive and equals
\begin{equation}
\rho_\nu(x)=\frac{i}{2\pi}\sqrt{Q^2_{1}(x)+4Q_{2}(x)}dx,
\label{eq:c1}
\end{equation}
where the value of $\sqrt{Q^2_{1}(x)+4Q_{2}(x)}$ is taken with negative
imaginary part. (The value of $Q^2_{1}(x)+4Q_{2}(x)$ on the
interval $[D,E]$ is negative.)

\begin{remark}
Note that $\int_{D}^E d\nu$ is not
necessarily equal to $1$. In fact, $\int_{D}^E d\nu=\frac{(E-D)^2}{16}.$
For comparison let us take the sequence of inverse ratios
$\frac{p_{n}(z)}{p_{n+1}(z)}$ for the above sequence of monic polynomials
$\{p_{n}(z)\}$.  Let $\tilde\nu_{n}$ denote the residue distribution
of $\frac{p_{n}(z)}{p_{n+1}(z)}$ and set $\tilde \nu=\lim_{n\to \infty
}\tilde \nu_{n}$. One can easily see that the distribution $\tilde
\nu$ is actually a probability  measure (i.e., it satisfies $\int_{D}^E
d\widetilde\nu=1$)
   having the same support as $\nu$ and its density $\rho_{\widetilde \nu}$ is  given by
$$
\rho_{\widetilde\nu}(x)=-\frac{i}{2\pi}\frac{\sqrt{Q^2_{1}(x)+4Q_{2}(x)}}{
Q_{2}(x) }dx.$$
\end{remark}

\begin{figure}
\centerline{\hbox{\epsfysize=2.5cm\epsfbox{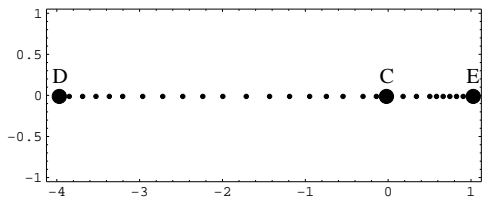}}}
\caption{Zeros of $p_{41}(z)$ satisfying the 3-term
recurrence relation $p_{n+1}(z)=zp_{n}(z)+(\frac
{3z-4}{4})p_{n-1}(z)$ with $p_{-1}(z)=0$ and $p_{0}(z)=1$}
\label{fig2}
\end{figure}

For the concrete example depicted in Figure~\ref{fig2} the asymptotic ratio
measure $\nu$ equals $\frac{2\sqrt{-x^2-3x+4}}{25\pi}dx$ and is supported on
the interval $[-4,1]$.
\end{case}

\begin{case}
The discriminant $\Xi_{\widetilde \phi}$ is the arc $D-C-E$ of the unique circle passing
through all three points $D,C,E$. In order to simplify the notation
while calculating $\nu$ assume that $\deg Q_{1}(z)=1$. Then by
an affine change of the variable $z$ we can normalize $Q_{1}(z)$ such that
$Q_{1}(z)=z$. Set $Q_{2}(z)=az^2+bz+c$, where $a,b,c$ are real numbers.
The condition that $Q_{1}^2(z)+4Q_{2}(z)=0$ has complex conjugate
  roots gives $b^2-c(4a+1)<0$. Then the center of the circle passing
through $D,C,E$ is given by $\big(\!-\frac{c}{b},0\big)$ and its radius is
$\big\vert \frac{c}{b}\big\vert$.
If $\gamma$ is the angle from the real positive half-axis to the ray through
the center of the circle and a point $(x,y)$ on
it then we have the following parametrizations: $x=\big\vert
\frac{c}{b}\big\vert\cos{\gamma}-\frac{c}{b}$, $y=\big\vert
\frac{c}{b}\big\vert\sin{\gamma}$ and $n=\cos{\gamma}+i\sin{\gamma}$ for the
unit normal to the circle at the point $(x,y)$.

In order to calculate the asymptotic ratio measure
$\nu=\frac{\partial\Psi_{max}(z)}{\partial \bar
z}$ notice 
that if $\Psi_{max}(z)$ is a piecewise
analytic function with smooth curves separating the domains where $\Psi_{max}(z)$
coincides with a given analytic function then $\nu=\frac{\partial
\Psi_{max}(z)}{\partial \bar z}$ is concentrated on these separation curves. Additionally, 
 at each smooth point where two branches $\Psi_{1}$ and
$\Psi_{2}$ meet $\nu$ satisfies the condition
$\nu=\frac {(\Psi_{1}-\Psi_{2})nds}{2\pi}$ (with an appropriate choice of
coorientation). Applying now this result to the case under consideration
we see that 
$$\nu=\frac{\partial\Psi_{max}(z)}{\partial \bar z}
=\frac{\sqrt{Q_{1}^2(z)+4Q_{2}(z)}n}{2\pi}\left\vert \frac{c}{b}\right\vert
d\gamma.$$

Explicit computations (using Mathematica$^{\textrm{TM}}$) give

\begin{equation}\label{eq:c2}
\nu=\frac{\sqrt 2}{b}(\cos{\gamma}+i\sin{\gamma})^{\frac 3
2}\sqrt{c(2b^2-(1+4a)c(1-\cos{\gamma}))}d\gamma.
\end{equation}

\begin{figure}
\centerline{\hbox{\epsfysize=5.5cm\epsfbox{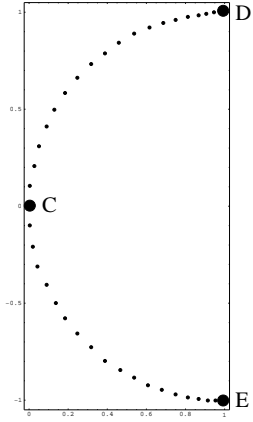}}}
\caption{Zeros of $p_{41}(z)$ satisfying the 3-term
recurrence relation $p_{n+1}(z)=zp_{n}(z)+(\frac {1-z}{2})p_{n-1}(z)$ with
$p_{-1}(z)=0$ and $p_{0}(z)=1$}
\label{fig3}
\end{figure}

For the concrete example considered in Figure~\ref{fig3} one has
$$\nu=\sqrt{2\cos\gamma}(\cos{\gamma}+i\sin{\gamma})^{\frac 3 2}d\gamma,$$
where $\gamma$ is the angle from the real positive half-axis to the ray
emanating from $\big(-\frac{c}{b},0)=(1,0)$ (i.e., the center of the
circle of radius $\big|\frac{c}{b}\big|=1$ containing the points
$D,C,E$ in this case) and passing through a variable point $(x,y)$ lying on
the left half-circle.
\end{case}

\begin{case}
The discriminant $\Xi_{\widetilde \phi}$ is the union of the interval $[D,E]$ and the
circle given (as in Case 2) by the equation $x(x-x_{0})+y^2=0$ where $x$
(respectively, $y$) is the real (respectively, imaginary) part of $z$ and
$x_{0}=-\frac {2c}{b}$.  The density $\rho_{\nu}$  of
$\nu=\frac{\partial\Psi_{max}(z)}{\partial \bar z}$ is given by
formula~\eqref{eq:c1} on the interval and by formula~\eqref{eq:c2} on the
circle. A concrete example is depicted in Figure~\ref{fig4} below.
\end{case}

\begin{figure}[!htb]
\centerline{\hbox{\epsfysize=3cm\epsfbox{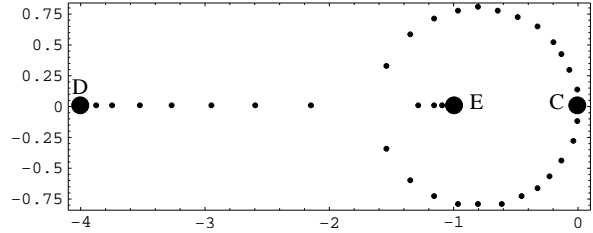}}}
\caption{Zeros of $p_{41}(z)$ satisfying the 3-term
recurrence relation $p_{n+1}(z)=zp_{n}(z)+(\frac {5z+4}{4})p_{n-1}(z)$ with
$p_{-1}(z)=0$ and $p_{0}(z)=1$}
\label{fig4}
\end{figure}

If either of the polynomials $Q_{1}(z)$ and $Q_{2}(z)$ has complex
coefficients it seems
difficult to give a more precise description of the support of $\nu$ than
the one obtained in Corollary~\ref{cor3}. Figure~\ref{fig5} below illustrates
 possible forms of this support.

\begin{figure}[!htb]
\centerline{\hbox{\epsfysize=5.5cm\epsfbox{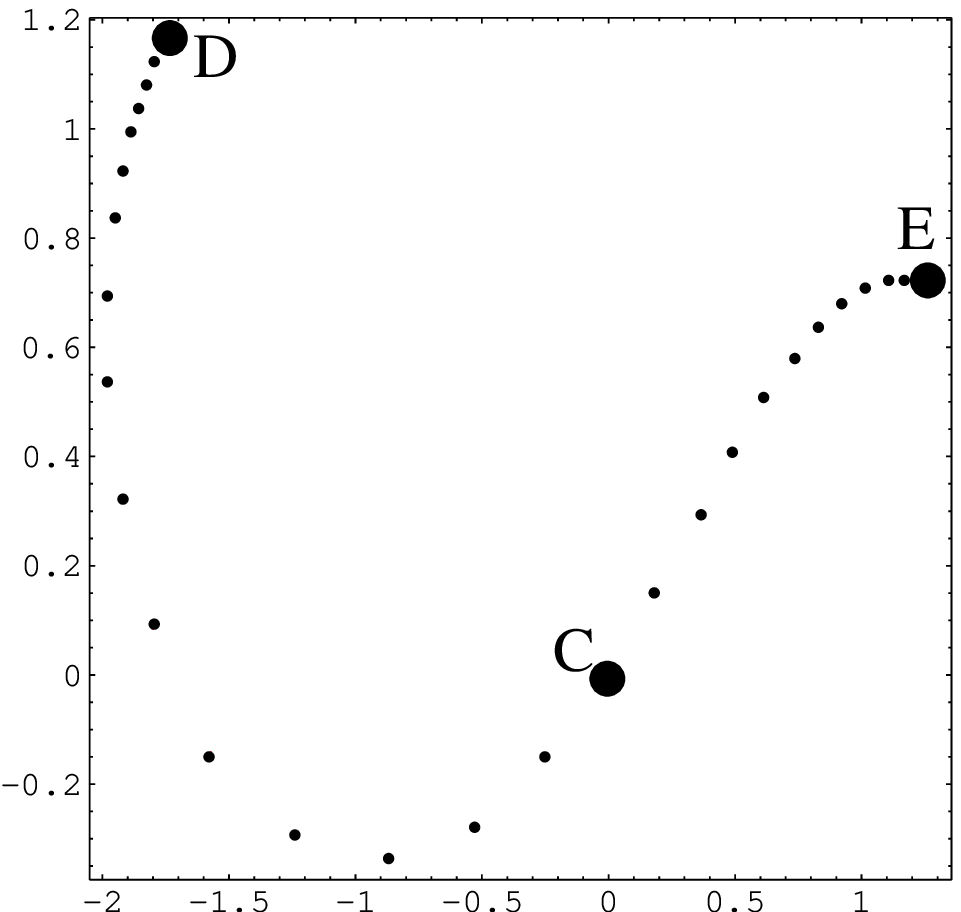}}
\hspace{0.5cm}\hbox{\epsfysize=5.5cm\epsfbox{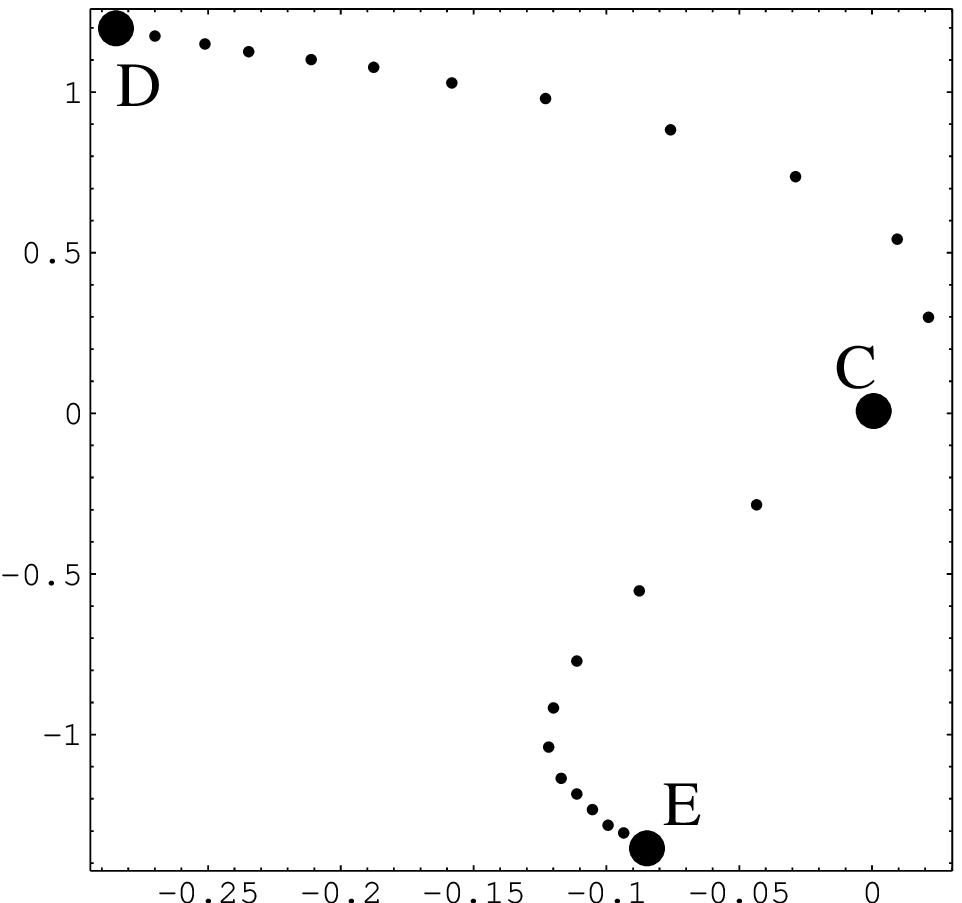}}}
\caption{Zeros of $p_{41}(z)$ for the recurrence relations
$p_{n+1}(z)=zp_{n}(z)+(iz^2+2z-1-3i)p_{n-1}(z)$ (left) and
$p_{n+1}(z)=zp_{n}(z)+((2-i)z^2+z+4-i)p_{n-1}(z)$ (right) with
$p_{-1}(z)=0$ and $p_{0}(z)=1$}
\label{fig5}
\end{figure}

Note that in all examples  in this
section we took the standard
initial polynomials $p_{-1}(z)=0$ and $p_{0}(z)=1$. Computer
experiments show no isolated zeros of $p_{n}(z)$ in these cases.
The following lemma shows that in the specific situation when $p_{-1}(z)=0$
and $p_{0}(z)=1$ the set $\Sigma_{I}$ of slowly growing initial conditions
(cf.~Definition~\ref{df:slow}) is actually empty.

\begin{lemma}
For any recurrence relation of the form~\eqref{genform} with initial values
$p_{-1}(z)=0$ and $p_0(z)=1$ the set $\Sigma_{I}$ of slowly growing initial
conditions is empty, comp. \cite{BSh}. 
\end{lemma}

\begin{proof}
A recurrence relation of the form~\eqref{genform} may be rewritten as
$$\begin{pmatrix}
p_{n+1}(z) \\
p_n(z)
\end{pmatrix}
=T(z)\begin{pmatrix}
p_n(z) \\
p_{n-1}(z)
\end{pmatrix},\,\text{ where }\,T(z)=
\begin{pmatrix}
Q_1(z) & Q_2(z)\\
1 & 0
\end{pmatrix}.$$
It is easily checked that for any $z\in \bC$ the vector $(1,0)^t\in \bC^2$
is not an eigenvector of $T(z)$, which immediately implies the desired result. \qed
\end{proof}

We should point out that as soon as one changes the initial pair
$(p_{-1}(z),p_{0}(z))$ the polynomials
$p_{n}(z)$ immediately acquire additional isolated zeros away from
$\Xi_{\widetilde \phi}$.

\section{Related topics and open problems}\label{s5}

\subsection{} Does Theorem~\ref{anconvprod}(ii) hold in the algebraic category, i.e. when the varying coefficients are algebraic functions of bounded degree?

\subsection{} Although one of the assumptions of  Theorem~\ref{th:Exist} is that all $\phi_{k,n}(\x)$ are non-vanishing in $\Omega$ computer experiments 
show that it holds in many cases where some $\phi_{k,n}(\x)$ vanish.  Apparently one can construct complicated counterexamples when the statement fails, in general. But it would be important for applications to find some sufficient conditions weaker  than $\phi_{k,n}(\x)\neq 0$ guaranteeing the validity of Theorem~\ref{th:Exist}(ii).  

\subsection{}  Theory of multiple orthogonal polynomials which experiences a rapid development at present is a natural area of application our results, see e.g. \cite {CCvA} and reference therein. Above we only considered a rather simple instance of biorthogonal polynomials. It would be nice to calculate the maxmod-discriminant $\Xi_{\widetilde \phi}$ and the asymptotic ratio distribution in more examples. 

\subsection{}
It is natural to ask what topological
information about the induced maxmod-discriminant may be
obtained from the coefficients of a given recurrence relation.

\begin{problem}
Describe the topological properties of $\Xi_{\widetilde \phi}$ depending on the
coefficients of the recurrence relation (say assumed fixed and algebraic) and
establish, in particular, necessary and sufficient conditions for the set
$\Xi_{\widetilde \phi}$ to be compact.
   Are there any characteristic numbers associated with the
   singularities on $\Xi_{\widetilde \phi}$?
\end{problem}





\section{Appendix. On topology and geometry of maxmod-discriminants}\label{s4}

\begin{definition}\label{def10}
Denote by $\widetilde \Xi_{k}\subset
Pol_{k}=\{t^k+a_{1}t^{k-1}+\ldots+a_{k}\}$ the set of
    all  monic polynomials of degree $k$ with complex coefficients
    having at least two roots with the same absolute value.
\end{definition}

It is obvious from Definitions~\ref{def5}, \ref{def7} and~\ref{def10} that
$\widetilde \Xi_{k}$ contains the standard maxmod-discriminant $\Xi_{k}$.

\begin{proposition}\label{propo4}
$\widetilde \Xi_{k}$ is a real semialgebraic hypersurface  
    of degree at most $(4k-1)(4k-2)$ in $ Pol_{k}$. Furthermore, the hypersurface 
$\widetilde \Xi_{k}$ is quasihomogeneous
    with quasihomogeneous weight equal to $i$ for both the real and the
imaginary parts of $a_{i}$,  $i=1,2,\ldots,k$. (Here $Pol_k$ is considered  as a 
$2k$-dimensional  real affine  space with the real and imaginary parts of all
    $a_{i}$'s chosen as coordinates.) 
\end{proposition}

\begin{proof}
In order to calculate the degree of $\widetilde \Xi_{k}$ let us describe
        an algorithm giving the equation for the analytic 
        continuation of $\widetilde \Xi_{k}$. This algorithm
may be presented as the  superposition of
    two resultants. Let $u=x+iy\in \bC$ and consider first the resultant of
    $P(t)=t^k+a_{1}t^{k-1}+\ldots+a_{k}$ and
    $P(ut)=(ut)^k+a_{1}(ut)^{k-1}+\ldots+a_{k}$, which we denote by
$R(u,a_{1},\ldots,a_{k})$. (Recall that all
    $a_{i}$'s are complex.)  Clearly,
    $R(u,a_{1},\ldots,a_{k})$ is a polynomial in the variable $u$. For any
    fixed value of $u$ the resultant $R(u,a_{1},\ldots,a_{k})$ vanishes if and
only if $P(t)$ and $P(ut)$ considered as polynomials in $t$ have a common
zero. One can easily see that $R(u,a_{1},\ldots,a_{k})$ is
    divisible by $(u-1)^k$ and that the quotient
$\widetilde R(u,a_{1},\ldots,a_{k})=R(u,a_{1},\ldots,a_{k})/(u-1)^k$ is coprime
    with $u-1$. We want to find an equation for the set of all monic complex
polynomials $P(t)$ such that
$P(ut)$ and $P(t)$ have a common zero for some $u\neq 1$ with
    $\vert u\vert=1$. The standard  rational  parametrization of the unit
    circle is given by $u=\frac{(1-i\theta)^2}{1+\theta^2}$,
    where $\theta\in \bR$.
    The result of the substitution
    $u=x+iy=\frac{(1-i\theta)^2}{1+\theta^2}$ in $\widetilde 
R(u,a_{1},\ldots,a_{k})$
    gives  a complex-valued rational function of the real variable
    $\theta$ whose  denominator equals to $(1+\theta^2)^{2k}$.
    By taking  the resultant of the real and imaginary
    parts of the numerator of the latter rational function one gets
    the required algebraic equation  for the analytic continuation of
$\widetilde \Xi_{k}$.
    This recipe allows us  to calculate (an upper bound for) the degree 
of $\widetilde \Xi_{k}$.
    It is easy to see that $\widetilde R(u,a_{1},\ldots,a_{k})$ is a
    polynomial of degree $2k$ in $u$ and of degree at most $2k-1$ in
    the variables $a_{1},\ldots,a_{k}$. After making the substitution
    $u=x+iy=\frac{(1-i\theta)^2}{1+\theta^2}$ and taking the real and
    imaginary parts of the above rational function one gets two
    polynomials of degrees at most $2k-1$ in the variables
    $\Re{a_{1}},\ldots,\Re{a_{k}},\Im{a_{1}},\ldots,\Im{a_{k}}$ and of 
degrees $4k$ and $4k-1$, respectively, in $\theta$. These
    polynomials have proportional leading terms and, therefore,  their resultant is
     a polynomial in the variables 
    $\Re{a_{1}},\ldots,\Re{a_{k}},\Im{a_{1}},\ldots,\Im{a_{k}}$ of
    degree at most $(2(4k)-2)(2k-1)=(4k-1)(4k-2)$.

The quasihomogeneity of $\widetilde \Xi_{k}$ with quasihomogeneous weights
    as specified in the statement of the proposition follows from the fact that
$\widetilde \Xi_{k}$ is
    preserved if one multiplies all roots of a polynomial by a
    non-negative real number. \qed
\end{proof}

\begin{remark}
The proof of Proposition~\ref{propo4} contains an algorithm giving the
explicit equation for the
analytic  continuation of $\widetilde \Xi_{k}$. The authors wrote a
Mathematica$^{\textrm{TM}}$ code doing this for small values of $k$. However, 
the resulting
expression contains several hundred terms even for $k=3$ and does not seem to
be of much use. For $k=3$ the estimate $(4k-1)(4k-2)$ for the degree
is sharp. It is very plausible that this estimate is actually sharp for 
arbitrary $k$ although in the general case it seems very
difficult to check the necessary nondegeneracy conditions while performing 
the superposition of the two resultants described in the above algorithm.
\end{remark}

\begin{proposition}\label{propo5}
$\widetilde \Xi_{k}$ is a singular real hypersurface with a
        (singular) boundary coinciding with the usual discriminant
	      $\mathcal D_{k}\subset Pol_{k}$, where   $\mathcal D_{k}$ is the set of all
polynomials in $Pol_{k}$ with multiple roots. Moreover, $\widetilde \Xi_{k}$
has no local singularities (i.e., singularities on a given
        branch) outside $\mathcal D_{k}$. The list of singularities of
$\,\widetilde \Xi_{k}$ is finite for any given $k\in \bN$, i.e. its  
singularities have no moduli.
\end{proposition}

\begin{proof}
Indeed, to show that $\mathcal D_{k}\subset
        Pol_{k}$ is the boundary of $\widetilde \Xi_{k}$ consider the
        standard Vieta map  $Vi:\bC^k \to Pol_{k}$ sending a $k$-tuple
        of (labeled) roots to the coefficients of the monic
        polynomial with these roots, i.e, to the elementary symmetric
        functions of these roots with alternating signs. It is known that the
Vieta map induces a local diffeomorphism
        on the complement $\bC^k\setminus \mathcal T_{k}$, where
        $\mathcal T_{k}$ is the standard Coxeter hyperplane arrangement
        consisting of $\binom{k}{2}$ hyperplanes $L_{i,j}$ given by
        $L_{i,j}: x_{i}=x_{j}$, $1\le i<j\le k$. An easy observation is that
$\mathcal T_{k}$
        coincides with the preimage $Vi^{-1}(\mathcal D_{k})$ of
        the discriminant $\mathcal D_{k}$.
        Consider now  the arrangement of quadratic cones
        $\mathfrak C_{k}=\cup_{i<j}\mathfrak{C}_{i,j}$ in $\bC^k$,
        where $\mathfrak{C}_{i,j}$ is given by the
        equation $\vert x_{i}\vert=\vert x_{j}\vert$.
        Obviously, $\mathfrak C_{k}$ coincides with the preimage
        $Vi^{-1}(\widetilde \Xi_{k})$.
        From the defining equation it is clear that  each
        $\mathfrak{C}_{i,j}$ is smooth outside the origin. Therefore,
$\widetilde \Xi_{k}$ is locally smooth (that is, it consists of smooth local
        branches) outside $\mathcal D_{k}$. Note also
        that for $1\le i<j\le k$ the quadratic cone $\mathfrak{C}_{i,j}$
contains the complex hyperplane
        $L_{i,j}$ and that the restriction $Vi\big|_{\mathfrak{C}_{i,j}}$ has a
fold near a generic point of $L_{i,j}$. Therefore,  $\mathcal D_{k}$
        is the boundary of $\widetilde \Xi_{k}$. The absence of the  moduli
        for the
        singularities of $\widetilde \Xi_{k}$  can be
        derived from that for the singularities of $\mathcal D_{k}$.
        The type of a (multi)singularity of
        $\widetilde \Xi_{k}$ near a
        polynomial $P(z)\in \widetilde \Xi_{k}$ is encoded
        in the following combinatorial information about the roots of
        $P(z)$. 
        Determine first the multiplicity (possibly vanishing) of the root of
        $P(z)$ at the origin and the set of
        all distinct positive absolute values for all the roots of $P(z)$.
        For each such positive absolute
value determine the number and multiplicity of   distinct roots having 
        this absolute value. If all the roots are simple
        then (as above) $\widetilde \Xi_{k}$ locally consists of
        (in general, nontransversal) smooth branches. Otherwise $P(z)$
        lies in $\mathcal D_{k}$, which is the boundary of $\widetilde
        \Xi_{k}$, and local branches of $\widetilde
        \Xi_{k}$ near $P(z)$ might have boundary singularities. \qed 
\end{proof}

Let $\Omega$ be an open subset of $\bR^2$ and assume that the coefficients of
the recurrence relation~\eqref{eq:General} are
  sufficiently generic. The arguments in the proof of Proposition~\ref{propo5}
applied to the induced maxmod-discriminant $\Xi_{\widetilde \phi}$ of the associated
symbol equation~\eqref{eq:Symb} imply that the only possible singularities on
the  curve $\Xi_{\widetilde \phi}$ are the endpoints and ``Y-type''
singularities, i.e., triple (local) rays emanating from a
given point, see Figures~\ref{fig1}--\ref{fig5}. Note however that the
transversal
intersection of smooth branches (that is, the interval and the circle) 
on Figure~\ref{fig4} is
actually unstable since it disappears under a small perturbation of the
coefficients of the recurrence relation.

Given a topological space $X$ let $\widehat X$ be its one-point
compactification.

\begin{proposition}\label{pr:hom}
The following properties hold for any $k\in \bN:$
\begin{enumerate}
\item[(i)] The hypersurfaces $\widetilde \Xi_{k}$ and $\Xi_{k}$ are
contractible in $Pol_{k}$, and, therefore,   their usual (co)homology groups are
trivial.
\item[(ii)] The one-point compactification $\widehat\Xi_{k}$ is
Alexander dual in the sphere $S^{2k}$ to a circle $S^1$. Therefore,
$\widehat\Xi_{k}$ is homotopically equivalent to a sphere $S^{2k-2}$, so
that $H_{i}\big(\widehat\Xi_{k},\mathbb Z\big)\cong \mathbb Z$ if
$i\in\{0,2k-2\}$ and $H_{i}\big(\widehat\Xi_{k},\mathbb Z\big)=0$ otherwise.
\item[(iii)] The one-point compactification $\widehat{\widetilde\Xi}_{k}$ is
Alexander dual in the sphere $S^{2k}$ to a $(k-1)$-dimensional
torus $\mathcal T^{k-1}$. Hence
$\widetilde{H}_{i}\big(\widehat{\widetilde\Xi}_{k},\mathbb Z\big)
\cong \widetilde{H}^{2k-i-1}\big(\mathcal T^{k-1},\mathbb Z\big)$ for any
$i$. Here $\widetilde {H}_{i}(X)$ (respectively, $\widetilde {H}^{i}(X)$)
stands for reduced mod point homology (respectively, cohomology) of the
topological space $X$, see e.g.~\cite{Va}.
\end{enumerate}
\end{proposition}

\begin{proof}
The contractibility of $\widetilde \Xi_{k}$ and $\Xi_{k}$
	   follows directly from their quasihomogeneity. In order to show (ii)
	   and (iii) we use the standard Alexander duality in $Pol_{k}$.
Let $\Upsilon_{k}= Pol_{k}\setminus \Xi_{k}$ and $\widetilde
	   \Upsilon_{k}= Pol_{k}\setminus \widetilde
	   \Xi_{k}$ denote the complements in $Pol_{k}$ of $\Xi_{k}$ and
$\widetilde \Xi_{k}$, respectively. Then
$$\widetilde H^{2k-i}(\Upsilon_{k},\mathbb Z)\cong \widetilde
	   H_{i}\big(\widehat {\Xi}_{k},\mathbb Z\big)\,
	  \text{ and }\,
	   \widetilde H^{2k-i}\big(\widetilde\Upsilon_{k},\mathbb Z\big)
\cong \widetilde H_{i}\big(\widehat{\widetilde\Xi}_{k},\mathbb Z\big).$$
Obviously, the space $\widetilde\Upsilon_{1}\simeq\Upsilon_{1}\simeq\bC$ is
contractible. The next lemma describes the topology of $\Upsilon_{k}$
and $\widetilde\Upsilon_{k}$ for $k>1$.

\begin{lemma}\label{lm:hom}
For any $k\in \bN$, $k\ge 2$, one has
\begin{enumerate}
\item[(i)] $\widetilde\Upsilon_{k}$ is an open $2k$-dimensional
manifold which is homotopically equivalent to the $(k-1)$-dimensional torus
$\mathcal T^{k-1}$.
\item[(ii)] $\Upsilon_{k}$ is an open $2k$-dimensional manifold
$\Upsilon_{k}$ which is homotopically equivalent to a circle $S^1$.
\end{enumerate}
\end{lemma}

\begin{proof}
(i) The space $\widetilde\Upsilon_{k}$ consists of
all $k$-tuples of complex numbers with distinct absolute
values. Let $X_{k}=\{r_{1}<r_{2}<\ldots<r_{k}\mid r_1\ge 0\}$ denote the
set of all possible $k$-tuples of distinct absolute values. Then
$\widetilde\Upsilon_{k}$ is ``fibered'' over $X_{k}$ with a
``fiber'' which is isomorphic to $\mathcal T^{k}$ if $r_{1}\neq 0$ and
isomorphic to $\mathcal T^{k-1}$ if $r_{1}=0$. In order to get
	       an actual fibration consider the set $\widehat
	       X_{k}=\{0<r_{2}<\ldots <r_{k}\}$ of the absolute values of the
	       roots starting from the second smallest.
Now $\widetilde\Upsilon_{k}$ is actually fibered over $\widehat X_{k}$ with
a fiber isomorphic to $\mathcal T^{k-1}\times D_{r_{2}}$, where
	       $D_{r_{2}}$ stands for the open disk of radius $r_{2}>0$
centered at the origin. The
observation that $\widehat X_{k}$ is contractible now implies that
$\widetilde\Upsilon_{k}$ is homotopically equivalent to $\mathcal T^{k-1}$.

(ii) The space $\Upsilon_{k}$ consists of
	       all $k$-tuples of complex numbers such there exists a unique
	       number with largest absolute value in the considered
	       $k$-tuple. Let $0<r_{max}$ denote this largest absolute
	       value. Then $\Upsilon_{k}$ is fibered over $\mathbb
	       R^+\simeq \{r_{max}\}$ with a fiber given by the product
$S^1\times Pol_{k-1}(r_{max})$, where $ Pol_{k-1}(r_{max})$ stands for
	       the set of all polynomials of degree $k-1$ whose roots
lie in the open disk of radius $r_{max}$ centered at the origin. Since both
$\mathbb R^+$ and $Pol_{k-1}(r_{max})$ are contractible it follows that
the space $\Upsilon_{k}$ is homotopically equivalent to $S^1$. \qed 
\end{proof}

Lemma~\ref{lm:hom} and the fact that $S^1$ is always unknotted in $Pol_{k}$
for $k\ge 2$ complete the proof of Proposition~\ref{pr:hom}. \qed
\end{proof}

\end{document}